\documentclass[a4paper,11pt,reqno]{amsart}

\usepackage{fullpage}

\usepackage[utf8]{inputenc} 
\usepackage[english]{babel}

\usepackage{dsfont,amscd}
\usepackage[cyr]{aeguill}
\usepackage[pdfdisplaydoctitle=true,
            colorlinks=true,
            urlcolor=blue,
            citecolor=blue,
            linkcolor=blue,
            pdfstartview=FitH,
            pdfpagemode=None,
            bookmarksnumbered=true]{hyperref}
\usepackage{amssymb}
\usepackage{multicol}
\usepackage{float}

\newtheorem{thm}{Theorem}

\newtheorem{lem}[thm]{Lemma}
\newtheorem{prop}[thm]{Proposition}
\theoremstyle{definition}
\newtheorem{defn}[thm]{Definition}
\newtheorem{hypo}[thm]{Hypothesis}
\theoremstyle{remark}
\newtheorem{rem}[thm]{Remark}

\numberwithin{equation}{section}

\newcommand{\set}[1]{\left\{#1\right\}}

\newcommand{\sldc}{\mathrm{SL}(2,\mathbb{C})}

\newcommand{\GC}{\mathbb{C}}

\newcommand{\uu}{\mathfrak{u}}
\newcommand{\ff}{\mathfrak{f}}
\newcommand{\bg}{\mathfrak{g}}
\newcommand{\bh}{\mathfrak{h}}
\newcommand{\ii}{\mathfrak{i}}
\newcommand{\cc}{\mathfrak{c}}
\newcommand{\cH}{\mathsf{H}}
\newcommand{\cU}{\mathsf{U}}
\newcommand{\cV}{\mathsf{V}}

\newcommand{\xx}{\mathbf{x}}

\DeclareMathOperator{\ord}{ord}
\DeclareMathOperator{\rank}{rank}

\newcommand{\tr}[3]{( #1,#2 )_{#3}}

\newcommand{\cov}[1]{\mathbf{Cov}(#1)}
\newcommand{\inv}[1]{\mathbf{Inv}(#1)}
\newcommand{\Sn}[1]{\mathrm{S}_{#1}}

\newcommand{\gset}{\mathcal{G}}

\hypersetup{
    pdfauthor={Reynald Lercier, Marc Olive}, 
    pdftitle={Covariant algebra of the binary nonic and the binary decimic}, 
    pdfsubject={}, 
    pdfkeywords={Classical invariant theory; Covariants; Gordan's algorithm}, 
    pdflang=en, 
    }

\AtBeginDocument{%

}

\begin{document}

\title{Covariant algebra of the binary nonic 
and the binary decimic}%

\author[Lercier]{Reynald Lercier}
\address{%
  \textsc{DGA MI}, %
  La Roche Marguerite, %
  35174 Bruz, %
  France. %
}
\address{%
  Institut de recherche math\'ematique de Rennes, %
  Universit\'e de Rennes 1, %
  Campus de Beaulieu, %
  35042 Rennes, %
  France. %
} 
\email{reynald.lercier@m4x.org}

\author[Olive]{Marc Olive}
\address{
LMA,%
Aix-Marseille Universit\'e,%
Centrale Marseille,%
13402 Marseille Cedex 20,%
France}
\email{marc.olive@math.cnrs.fr}%
	
\subjclass[2010]{13C99,14Q99}%
\keywords{Classical invariant theory; Covariants; Gordan's algorithm}%
\date{\today}

\begin{abstract}
  We give a minimal system of 476 generators (resp. 510 generators) for the
  algebra of $\mathrm{SL}_2(\mathbb{C})$-covariant polynomials on binary forms
  of degree 9 (resp. degree 10). These results were only known as conjectures
  so far.  The computations rely on Gordan's algorithm, and some new
  improvements.
\end{abstract}

\maketitle

\section{Introduction}
\label{sec:intro}

\emph{Invariant theory} regularly comes up for discussion with numerous
attempts to obtain new results. After the important Weyl's
contribution~\cite{Wey1997} in the field of \emph{representation theory}, many
other reformulations have been made on the subject, as the ones of
Dieudonn\'e~\cite{DC1970}, Kung--Rota~\cite{KR1984} or
Howe~\cite{How1989,How1989a}. Most of them being of theoretic interest, the
emergence of computer science revives interest in effective approaches, with
the hope that new results could suddenly be reached. Besides, effective
approaches also appear to have many important applications, as in continuum
mechanics~\cite{RE1955,SR1971,Boe1987,BKO1994}, quantum
informatics~\cite{Luq2007}, recoupling
theory~\cite{AC2007,AC2009,Abd2012,AC2012}, cohomology of finite
groups~\cite{AM2004}, computation of Galois groups~\cite{Sta1973,GK2000},
cryptography~\cite{FS2012,FS2013,FGHR14}, or
combinatorics~\cite{Sta1979,Sta1979b}.\medskip

Classical invariant theory\footnote{We can found an important literature on this subject made in the field of science sociology~\cite{Fis1966} or history of science~\cite{Cri1986,Cri1988,Par1989}.}, which is somewhat the cradle of invariant theory, was first initiated by Boole~\cite{Boo1841}. After this work, two different teams worked on the subject: an English one leaded by Cayley, Sylvester \textit{et al.}~\cite{Syl1878,Cay1861,Cay2009} and a German one leaded by Clebsch, Gordan \textit{et al.}~\cite{CG1967,Gor1987,Gor1900,Gor1900a,Gor1875,Gor1868}. The first finiteness result, obtained by Gordan~\cite{Gor1868} in the case of $\sldc$ invariants of binary forms, was closely endowed with a constructive proof, namely Gordan's algorithm on binary forms~\cite{GY2010,Wey1993,Olive2014}. As an application, Gordan and von Gall~\cite{vGal1874,Gor1875,vGal1888} obtained some non trivial finite \emph{covariant basis} of binary forms: the ones of quintic~\cite{Gor1875,GY2010}, sextic~\cite{Gor1875,GY2010}, septimic~\cite{Gor1875,vGal1888} and octics~\cite{vGal1880}. All those computations were ``hand made'', and up to nowadays, except computations made on octics~\cite{Croe2002,Bed2008} or septimics~\cite{Bed2009}, no any new results had been obtained in the topics of covariant basis of a single binary form.  \medskip

Most of recent works on the
subject~\cite{DC1970,KR1984,How1989a,Der1999,DK2002,DK2008,Kem2003,Stu2008}
make use of \emph{algebraical geometry} tools, mainly developed by Hilbert
himself~\cite{Hil1993}. One important step on this way is to obtain an
homogeneous system of parameters (h.s.o.p), which gives degree upper-bounds on
generators thanks to the \emph{Hilbert series} of the invariant algebra. But
calculating such a h.s.o.p is often difficult. Up to our knowledge, there is no
general algorithm for this task, despite some recent attempts on that
subject~\cite{Has2008}.

Nevertheless, in the case of binary forms, many theoretical results on h.s.o.p
and Hilbert series have been obtained by
Dixmier~\cite{Dix1981,Dix1982,Dix1985,Dix1987}. In addition, Brouwer and
Popoviciu~\cite{BP2010a,BP2010,BP2011,BP2012, POP2014} made
important progress for nonics and decimics. For the first time,
explicit h.s.o.p and minimal invariant basis were obtained for these spaces of
binary forms.

We present in this article new results on covariant basis, which rely on mixed
ideas coming from the classical algebraical geometry
approach~\cite{DC1970,Stu2008,BP2010}, some recent works made on \emph{linear
  Diophantine equation}~\cite{CF1990,PV2004} and a Gordan's algorithm
reformulation\footnote{Note that Weyman~\cite{Wey1993} has also reformulated
  Gordan's method in a modern way through algebraic geometry but,
  unfortunately, we were unable to extract from it an effective
  approach. There is also a preprint by Pasechnik~\cite{Pas1996} on this
  method.}~\cite{Olive2014a}. We follow in a way Kung--Rota's
remark~\cite{KR1984}, ``After Hilbert's work, Gordan's ideas were
abandoned. However, Gordan's method remains the most effective one''. We show
with this approach that there exists a minimal covariant basis with 476
generators for the binary nonic (Theorem~\ref{thm:covar-basis-binary-9}),
respectively with 510 generators for the binary decimic
(Theorem~\ref{thm:covar-basis-binary-10}).  We point out that those results
have long been conjectured~\cite{Brouwer2015}.\bigskip

The paper is organized as follows. In \autoref{sec:Math_Frame}, we give some
mathematical backgrounds related to classical invariant theory of binary
forms: definition of the invariant and covariant algebra, Cohen--Macaylayness
property, Hilbert series, h.s.o.p. \textit{etc.} In \autoref{sec:Gordan_Alg},
we introduce Gordan ideals, which are the cornerstone of Gordan's algorithm,
presented subsequently in \autoref{sec:The_Alg_of_Gordan}. Then,
\autoref{sec:impr-gord-algor} focuses on some important improvements on
Gordan's algorithm and \autoref{sec:computation} gives details on calculation
strategies that yield our new results. Finally, \autoref{sec:results} presents
all the computations and results obtained for the covariant algebra of binary
nonics and decimics. For the sake of completeness, we explicit their minimal
covariant basis in \autoref{sec:mathbfcov_9} and \autoref{sec:mathbfcov_10}.


\section{Mathematical framework}\label{sec:Math_Frame}

\subsection{Covariants of binary forms}
\label{sec:AlgCovariants}
 
The complex vector space of $n$-th degree binary forms, denoted $\Sn{n}$, is the space of homogeneous polynomials
\begin{displaymath}
	\ff(\xx)=a_0x^n+a_1x^{n-1}y+\dotsc +a_{n-1}xy^{n-1} +a_ny^n
\end{displaymath}
with $\xx=(x,y)\in \mathbb{C}^{2}$ and $a_i\in \mathbb{C}$.
The natural $\mathrm{SL}_2(\mathbb{C})$ action on $\mathbb{C}^2$ induces a left action on $\Sn{n}$, given by
\begin{displaymath}
	(g\cdot \ff)(\xx):=\ff(g^{-1}\cdot \xx) \text{ for } g\in \mathrm{SL}_2(\mathbb{C})
        \,.
\end{displaymath}
More generally, by a space $V$ of binary forms, we mean a direct sum
\begin{displaymath}
	V:=\bigoplus_{i=0}^s \Sn{n_i}
\end{displaymath}
where the action of $\mathrm{SL}_2(\mathbb{C})$ is diagonal.

The invariant algebra of $V$, denoted $\inv{V}$, is the algebra
\begin{math}
  \inv{V}:=
  \mathbb{C}[V]^{\mathrm{SL}_2(\mathbb{C})}
  \,.
\end{math}
An important result, first established by Gordan~\cite{Gor1868}, and then extended by Hilbert~\cite{Hil1993} to any linear reductive group, is the following.
\begin{thm}\label{thm:covar-binary-forms}
For every space $V$ of binary forms, the algebra $\inv{V}$ is finitely generated, i.e. there exists a finite set $\lbrace\ii_1,\dotsc,\ii_s\rbrace$ in $\inv{V}$, called a \emph{basis}, such that
\begin{displaymath}
	\inv{V}=\mathbb{C}[\ii_1,\dotsc,\ii_s].
\end{displaymath}
\end{thm}\medskip

The covariant algebra of a space $V$ of binary forms, denoted $\cov{V}$, is the invariant algebra
\begin{displaymath}
	\cov{V}:=\mathbb{C}[V\oplus \mathbb{C}^2]^{\mathrm{SL}_2(\mathbb{C})}\,
\end{displaymath}
with the action of $\mathrm{SL}_2(\mathbb{C})$ on $\mathbb{C}[V\oplus \mathbb{C}^2]$ defined by
\begin{displaymath}
  (g\cdot p)(\ff,\xx):=p(g^{-1}\cdot \ff,g^{-1}\cdot \xx)  \text{ for } g\in \mathrm{SL}_2(\mathbb{C}), \, p\in \mathbb{C}[V\oplus \mathbb{C}^2]\,.
\end{displaymath}
Similarly to $\inv{V}$, the algebra $\cov{V}$ is finitely generated.\medskip

There is a natural bi-gradation on the covariant algebra $\cov{V}$,
\begin{itemize}
\item by the \textbf{degree} $d$, the polynomial degree in the coefficients of the space $V$,
\item and by the \textbf{order} $m$, the polynomial degree in the variables
  $x$, $y$.
\end{itemize}

We know an important upper-bound on the \emph{order} of
generators.  For every integer $n$, we take $\lambda$ to be the maximal
integer such that $n=2^{\lambda}+\nu$ and we define
\begin{equation}\label{eq:1}
	\lambda_n:=(\lambda-1)2^{\lambda}+\nu(\lambda+1)+2.
\end{equation}
Then we have this fact.

\begin{lem}[\cite{GY2010}]\label{lem:Order_Bounds_GY}
For every space $V=\bigoplus_{i=0}^s \Sn{n_i}$ ($n_0\leqslant \dotsc \leqslant n_s$) of binary forms, the covariant algebra $\cov{V}$ is generated by covariants of maximum order $\lambda_{n_s}$. 
\end{lem}

As a corollary, the covariant algebra $\cov{\Sn{9}}$ (resp. $\cov{\Sn{10}}$) is generated by covariants of maximum order $22$ (resp. $26$).   

We now focus on \emph{minimal basis} of covariant algebras. Take a space $V$ of binary forms and define $\mathbf{Cov}_{d,m}(V)$ to be the subspace of degree $d$ and order $m$ covariants. Now let
\begin{displaymath}
  \mathrm{C}_{+}:=\sum_{d+m>0} \mathbf{Cov}_{d,m}(V)
\end{displaymath}
which is an ideal of the graded algebra $\cov{V}$. For each $(d,m)$ such that $d+m>0$, let
$\delta_{d,m}$ be the codimension of $(\mathrm{C}_{+}^2)_{d,m}$ in
$\mathbf{Cov}_{d,m}$. Since the algebra $\cov{V}$ is of finite type, there
exists an integer $p$ such that $\delta_{d,m}=0$ for $d+m\geqslant p$ and we can
define the invariant number $n(V)$:
\begin{displaymath}
	n(V):=\sum_{d,m} \delta_{d,m} .
\end{displaymath}

\begin{defn}
A family $\set{\cc_{1},\dotsc,\cc_{s}}$ is a \emph{minimal} basis of $\cov{V}$ if its image in the vector space $\mathrm{C}_{+}/\mathrm{C}_{+}^2$ is a basis. In that case, we have $s=n(V)$\,.
\end{defn}

\begin{rem} 
As pointed out by Dixmier--Lazard~\cite{DL1985/86}, a minimal basis is
obtained by taking, for each couple $(d,m)$ a complement basis of
$(\mathrm{C}_{+}^2)_{d,m}$ in $\mathbf{Cov}_{d,m}(V)$. There is a long history of an
explicit determination of such a minimal basis for covariant algebras. We give
in~\autoref{tab:Min_Bas} some results (see~\cite{Brouwer2015} for a general
overview).
As far as we know, there is no way to obtain the invariant number $n(V)$ but to
exhibit a minimal basis.
\begin{table}[H]
\begin{tabular}{|c|c|c|}
\hline 
Algebra & $n(V)$ & Explicit minimal basis \\ 
\hline 
\hline
$\cov{\Sn{5}}$ & $23$ & Gordan~\cite{Gor1868} \\ \hline 
$\cov{\Sn{6}}$ & $26$ & Gordan~\cite{Gor1868} \\ \hline
$\cov{\Sn{6}\oplus\Sn{2}}$ & $99$ & Von Gall~\cite{vGal1888} \\ \hline
$\cov{\Sn{7}}$ & $147$ & Dixmier--Lazard~\cite{DL1985/86}, Bedratyuk~\cite{Bed2009} \\ \hline
$\cov{\Sn{8}}$ & $69$ & \begin{tabular}{c} Cr\"oni~\cite{Croe2002}, Bedratyuk~\cite{Bed2008} \\
Popoviciu~\cite{POP2014} \end{tabular}\\ \hline
$\cov{\Sn{9}}$ & $476$ & This paper \\ \hline
$\cov{\Sn{10}}$ & $510$ & This paper \\ \hline
\hline 
\end{tabular} \smallskip
\caption{Minimal basis of covariant algebras}\label{tab:Min_Bas}
\end{table}
\end{rem}

\subsection{Cayley's operator and transvectants}

To calculate covariants, we make use of the \emph{Cayley's operator}, defined
on a tensor product $\Sn{n}\otimes \Sn{m}$ (seen as a tensor product of
complex analytic functions) by
\begin{displaymath}
\Omega_{\alpha\beta}(\ff(\xx_{\alpha})\bg(\xx_{\beta})):=\frac{\partial \ff}{x_{\alpha}}\frac{\partial \bg}{y_{\beta}}-\frac{\partial \ff}{y_{\alpha}}\frac{\partial \bg}{x_{\beta}},\quad \ff \in \Sn{n},\quad \bg \in \Sn{m}\,.
\end{displaymath}

\begin{defn}
Given two binary forms $\ff\in \Sn{n}$ and $\bg\in \Sn{m}$, their \emph{transvectant} of index $r\geqslant 0$, denoted $\tr{\ff}{\bg}{r}$, is defined to be
\begin{displaymath}
	\tr{\ff}{\bg}{r}:=\begin{cases}
	\displaystyle{\frac{(n-r)!}{n!}\frac{(m-r)!}{m!}\mu\circ
	\Omega_{\alpha\beta}^r(\ff(\xx_{\alpha})\bg(\xx_{\beta}))}\ \text{ if } 0\leqslant r\leqslant \min (n,m)\,, \\
	0 \text{ otherwise,}
	\end{cases}
\end{displaymath}
where $\mu$ is a trace operator, $\mu(\bh(\xx_{\alpha},\xx_{\beta})):=\bh(\xx,\xx)$.
\end{defn}

\begin{rem}
There exists many other but equivalent definition of the transvectant, related to
\emph{group theory representation}. Indeed, $\sldc$ irreducible
representations are given by spaces $\Sn{n}$ of binary forms. By
Clebsch--Gordan decomposition, we have
\begin{displaymath}
	\Sn{n}\otimes \Sn{m}\simeq \bigoplus_{r=0}^{\min(n,m)} \Sn{n+m-2r}\,,
\end{displaymath}
and the unique projection (up to a scale factor) from $\Sn{n}\otimes \Sn{m}$
to $\Sn{n+m-2r}$ is the transvectant.
\end{rem}

\subsection{Cohen-Macaulayness}

We focus now on classical results issued from commutative algebra.
We refer the interested reader to some general
books~\cite{Lan2002,YOS1990,Eis1995,BH1998}.

Let $\mathcal{R}$ be a finitely generated graded $\GC$--algebra,
\begin{equation}\label{formule:Alg_graduee}
	\mathcal{R}=\bigoplus_{i\geqslant 0} \mathcal{R}_{i},\quad \mathcal{R}_{i}\mathcal{R}_{j}\subset \mathcal{R}_{i+j}.
\end{equation}
A finite family $\theta_{1},\cdots,\theta_{s}$ of free algebraic elements is a
\emph{homogeneous system of parameters} (h.s.o.p) if every element is
homogeneous and if the algebra $\mathcal{R}$ is a
$\GC[\theta_{1},\cdots,\theta_{s}]$-module of finite type. The number $s$ is
nothing else than the \emph{Krull dimension}~\cite{Kna2007} of
$\mathcal{R}$. From the Noether normalization Lemma~\cite{Lan2002}, a h.s.o.p
always exists for a finitely generated ring. Nevertheless, this result is not
constructive: up to our knowledge, there is no general algorithm to obtain a
h.s.o.p, although some papers initiated the subject~\cite{Has2008}. 

The algebra $\mathcal{R}$ is said to be \emph{Cohen--Macaulay} if it is a \emph{free} $\GC[\theta_{1},\cdots,\theta_{s}]$-module of finite type. In that case, there exists elements $\eta_{1},\dotsc,\eta_{r}$ such that
\begin{equation}\label{formule:Hironaka_Dec}
	\mathcal{R}=\eta_{1}\GC[\theta_{1},\cdots,\theta_{s}]\oplus \dotsc \oplus\eta_{r}\GC[\theta_{1},\cdots,\theta_{s}]\,.
\end{equation}
This direct sum is called the \emph{Hironaka decomposition} of $\mathcal{R}$.

In invariant theory (for linear reductive group), an invariant algebra
$\mathcal{R}$ is always Cohen--Macaulay~\cite{HR1974}, especially $\inv{V}$ is
Cohen--Macaulay.\smallskip

Take now $\mathcal{M}$ to be a finitely generated graded $\mathcal{R}$--module
and take again $\theta_{1},\cdots,\theta_{s}$ to be a h.s.o.p for
$\mathcal{R}$. When the module $\mathcal{M}$ is Cohen--Macaulay, we know that
$\mathcal{M}$ is a \emph{free}
$\GC[\theta_{1},\cdots,\theta_{s}]$--module. Thus there exists
$\mathfrak{m}_1,\dotsc,\mathfrak{m}_p\in \mathcal{M}$ such that a Hironaka
decomposition of $\mathcal{M}$ is
\begin{equation}\label{formule:Hironaka_Dec_Mod}
  \mathcal{M}=\mathfrak{m}_{1}\GC[\theta_{1},\cdots,\theta_{s}]\oplus \dotsc \oplus\mathfrak{m}_{p}\GC[\theta_{1},\cdots,\theta_{s}]\,.
\end{equation}
For a \emph{covariant algebra} $\cov{V}$, let us observe that for every
integer $m>0$, the space $\mathbf{Cov}_{m}(V)$ of $m$-th order covariants is a
$\inv{V}$--module. \medskip

We have an important result due to Van Den Bergh~\cite{MR1128219,Van1994}. 
For every integer $n$, let us define 
\begin{displaymath}
	\sigma_n:=
	\begin{cases}
		\cfrac{(n+1)^2}{4}\ \text{ if } n \text{ is odd}\,, \\	
		\cfrac{n(n+2)}{4}\ \text{ otherwise}\,.
	\end{cases}
\end{displaymath}
Take now $V$ to be the space of binary forms $\bigoplus_{i=0}^s \Sn{n_i}$ and let $\sigma_V$ be
\begin{math}
	\sum_{i=0}^{s} \sigma_{n_i},
\end{math}
we can state the following theorem.
\begin{thm}\label{thm:cohen-macaulayness}
For every integer $m<\sigma_V-2$, the $\inv{V}$--module $\mathbf{Cov}_{m}(V)$ of $m$-th order covariants is Cohen--Macaulay.
\end{thm}
As a corollary, the $\inv{\Sn{9}}$--module $\mathbf{Cov}_{m}(\Sn{9})$ is Cohen--Macaulay for every integer $m<25$ and the $\inv{\Sn{10}}$--module $\mathbf{Cov}_{m}(\Sn{10})$ is Cohen--Macaulay for every integer $m<30$.  \medskip

We now exhibit h.s.o.p. for $\inv{\Sn{9}}$ and $\inv{\Sn{10}}$. Write first $\ff\in \Sn{9}$ and 
\begin{align*}
\mathfrak{h}_1&:=\tr{\ff}{\ff}{8}\in \Sn{2},\quad \mathfrak{h}_2:=\tr{\ff}{\ff}{6}\in \Sn{6},\quad \mathfrak{h}_3:=\tr{\ff}{\ff}{4}\in \Sn{10},\quad \mathfrak{h}_4:=\tr{\ff}{\ff}{2}\in \Sn{14},\\
\mathfrak{h}_5&:=\tr{\ff}{\mathfrak{h}_2}{6}\in \Sn{3},\quad \mathfrak{h}_6:=\tr{\ff}{\mathfrak{h}_5}{3}\in \Sn{6},\quad \mathfrak{h}_7:=\tr{\ff}{\mathfrak{h}_5}{1}\in \Sn{10},\quad \mathfrak{h}_8:=\tr{\mathfrak{h}_{2}}{\mathfrak{h}_{2}}{4}\in \Sn{4}, \\
\mathfrak{h}_{9}&:=\tr{\mathfrak{h}_{5}}{\mathfrak{h}_{5}}{2}\in \Sn{2},\quad \mathfrak{h}_{10}:=\tr{\mathfrak{h}_{8}}{\mathfrak{h}_{9}}{0}\in \Sn{6},\quad \mathfrak{h}_{11}:=\tr{\mathfrak{h}_{8}}{\mathfrak{h}_{9}}{1}\in \Sn{4}.
\end{align*}

\begin{prop}[\cite{Dix1985,BP2010a}]\label{prop:hsop_Inv9}
The algebra $\inv{\Sn{9}}$ has a homogeneous system of parameters of degrees
$4$, $4$, $8$, $12$, $14$, $16$ and $30$ given by 
\begin{align*}
\mathfrak{p}_4&:=\tr{\mathfrak{h}_1}{\mathfrak{h}_1}{2},\quad \mathfrak{q}_{4}=\tr{\mathfrak{h}_{2}}{\mathfrak{h}_{2}}{6},\quad \mathfrak{p}_{8}:=\tr{\mathfrak{h}_{1}^{3}}{\mathfrak{h}_{2}}{6},\quad \mathfrak{p}_{12}:=\tr{\mathfrak{h}_{1}^{5}}{\mathfrak{h}_{3}}{10}, \\
\mathfrak{p}_{14}&:=\tr{\mathfrak{h}_{1}^{5}}{\mathfrak{h}_{7}}{10},\quad \mathfrak{p}_{16}:=\tr{\mathfrak{h}_{1}^{7}}{\mathfrak{h}_{4}}{14},\quad \mathfrak{p}_{30}:=\tr{\tr{\mathfrak{h}_{10}}{\mathfrak{h}_{10}}{4}}{\mathfrak{h}_{11}}{4}.
\end{align*}
The algebra $\inv{\Sn{9}}$ has also homogeneous systems of parameters of
degrees $4$, $8$, $10$, $12$, $12$, $14$, $16$, degrees  $4$, $4$, $10$, $12$, $14$,
$16$, $24$, degrees $4$, $4$, $8$, $10$, $12$, $16$, $42$ and degrees $4$, $4$, $8$, $10$,
$12$, $14$, $48$.
\end{prop}

Now let $\ff\in \Sn{10}$ and 
\begin{align*}
\mathfrak{h}'_1&:=\tr{\ff}{\ff}{8}\in \Sn{4},\quad \mathfrak{h}'_{2}:=\tr{\ff}{\mathfrak{h}'_1}{4}\in \Sn{6},\quad \mathfrak{h}'_{3}:=\tr{\ff}{\ff}{6}\in \Sn{8},\quad \mathfrak{h}'_{4}:=\tr{\mathfrak{h}'_{3}}{\ff}{8}\in \Sn{2}, \\
\mathfrak{h}'_{5}&:=\tr{\mathfrak{h}'_{3}}{\mathfrak{h}'_{3}}{8}\in \Sn{4},\quad \mathfrak{h}'_{6}:=\tr{\mathfrak{h}'_{2}}{\mathfrak{h}'_{2}}{4}\in \Sn{4},\quad \mathfrak{h}'_{7}:=\tr{\mathfrak{h}'_{3}}{\mathfrak{h}'_{5}}{4}\in \Sn{4}.
\end{align*}

\begin{prop}[\cite{BP2010}]\label{prop:hsop_Inv10}
The algebra $\inv{\Sn{10}}$ has a homogeneous system of parameters of degrees
$2$, $4$, $6$, $6$, $8$, $9$, $10$ and $14$ given by 
\begin{align*}
\mathfrak{p}'_2&:=\tr{\ff}{\ff}{10},\quad \mathfrak{p}'_4:=\tr{\mathfrak{h}'_1}{\mathfrak{h}'_1}{4},\quad \mathfrak{p}'_6:=\tr{\mathfrak{h}'_{2}}{\mathfrak{h}'_{2}}{2},\quad \mathfrak{q}'_6:=\tr{\mathfrak{h}'_{4}}{\mathfrak{h}'_{4}}{2}, \\
\mathfrak{p}'_8&:=\tr{\mathfrak{h}'_1}{\mathfrak{h}'_{6}}{4},\quad \mathfrak{p}'_9:=\tr{\tr{\mathfrak{h}'_{2}}{\mathfrak{h}'_1}{1}}{\mathfrak{h}^{'2}_1}{8},\quad 
\mathfrak{p}'_{10}:=\tr{\tr{\mathfrak{h}'_{2}}{\mathfrak{h}'_{2}}{2}}{\mathfrak{h}^{'2}_1}{8},\\
\mathfrak{p}'_{14}&:=\tr{\tr{\mathfrak{h}'_{5}}{\mathfrak{h}'_{5}}{2}}{\mathfrak{h}'_{7}}{4}+\tr{\tr{\mathfrak{h}'_1}{\mathfrak{h}'_1}{2}^2}{\tr{\mathfrak{h}'_{2}}{\mathfrak{h}'_{2}}{2}}{8}.
\end{align*}
\end{prop}

\subsection{Hilbert series and degree upper-bounds}
\label{sec:hilbert-series}

Let $\mathcal{M}:=\bigoplus_{i\geqslant 0} \mathcal{M}_{i}$ be a graded $\mathcal{R}$--module,
its Hilbert series is defined to be
\begin{displaymath}
  \mathcal{H}_{\mathcal{M}}(z):=\sum_{i\geqslant 0} \dim (\mathcal{M}_i)z^{i}\,.
\end{displaymath} 
A classical result states that the Hilbert series of a Cohen-Macaulay module
with Hironaka decomposition given by \eqref{formule:Hironaka_Dec_Mod} is
\begin{displaymath}
	\mathcal{H}_{\mathcal{M}}(z):=\frac{z^{m_{1}}+\dotsc+z^{m_{p}}}{(1-z^{d_{1}})\dotsc (1-z^{d_{s}})}\,,
\end{displaymath}
where $m_{i}$ is the degree of $\mathfrak{m}_{i}$ and $d_{j}$ is the degree of
$\theta_{j}$.

If the family $\theta_{1},\dotsc,\theta_{s}$ is a system of parameters, each subfamily $\theta_{1},\dotsc,\theta_{j}$ $(j\leqslant s)$ is a regular sequence and, writing 
\begin{displaymath}
	\overline{\mathcal{M}}:=\mathcal{M}/(\theta_{1}\mathcal{M}+\dotsc+\theta_{j}\mathcal{M})\,,
\end{displaymath}
we have
\begin{equation}\label{eq:Hilbert_S_Quotient}
	\mathcal{H}_{\overline{\mathcal{M}}}(z)=(1-z^{d_{1}})\dotsc(1-z^{d_{j}})\mathcal{H}_{\mathcal{M}}(z)\,.
\end{equation}
\medskip

In our  case of interest, \textit{i.e.} a covariant algebras of binary forms
\begin{displaymath}
	\cov{V}=\bigoplus_{d,m\geqslant 0} \mathbf{Cov}_{d,m} (V)\,,
\end{displaymath}
we make use of the \emph{multi-graded} Hilbert series. Let $a_{d,m}:=\dim
(\mathbf{Cov}_{d,m} (V))$, we can define
\begin{displaymath}
	\mathcal{H}_{\cov{V}}(t,z):=\sum_{d,m\geqslant 0} a_{d,m}t^{m}z^{d}.
\end{displaymath}
The Hilbert series of $\inv{V}$ can be easily obtained
from the \emph{multi-graded} Hilbert series of $\cov{V}$,
\begin{displaymath}
	\mathcal{H}_{\inv{V}}(z)=\mathcal{H}_{\cov{V}}(0,z).
\end{displaymath}
More generally, we can deduce the Hilbert series of the $\inv{V}$--module
$\mathbf{Cov}_{m} (V)$ of $m$-th order covariants,
\begin{displaymath}
	\mathcal{H}_{\mathbf{Cov}_{m} (V)}(z)=\sum_{d\geqslant 0} a_{d,m}z^{d}.
\end{displaymath}

Finally, note that there exists many ways to compute such series \emph{a
  priori}~\cite{SF1879,Spr1983,LP1990}, especially Bedratyuk's
developed a \textsc{maple} package~\cite{Bed2011}. For a direct computation of a given
$a_{d,m}$, we have this nice formula too.
\begin{thm}[Springer~\cite{Spr1983}]\label{thm:springer}
  The dimension of $\mathbf{Cov}_{d,m} (\Sn{n})$ is equal to the $\lfloor ({nd - m})/{2}
  \rfloor$-th coefficient of the power series expansion of
  \begin{displaymath}
    \frac{(1-q^n)\,(1-q^{n+1})\ldots (1-q^{n+d})}{(1-q^{2})\ldots (1-q^{d})}\,.
  \end{displaymath}
\end{thm}

When a h.s.o.p is known for a Cohen--Macaulay $\inv{V}$--module $\mathbf{Cov}_{m} (V)$, we directly deduce from its Hilbert series some upper-bounds on generator degrees. 
\begin{lem}\label{lem:degree_Bound_S9}
The $\inv{\Sn{9}}$--module of m-th order covariant $\mathbf{Cov}_{m} (\Sn{9})$ is generated by covariants of maximum degree $d_m$ given in the following table, for $m\leqslant 22$:\smallskip
\begin{displaymath}\setlength{\arraycolsep}{2pt} 
  \begin{array}{|c||c|c|c|c|c|c|c|c|c|c|c|c|c|c|c|c|c|c|c|c|c|c|c|}\hline
    \mathrm{Ord.}\ m& 0 & 1 & 2 & 3 & 4 & 5 & 6 & 7 & 8 & 9 & 10 & 11 & 12 & 13 & 14 & 15 & 16 & 17 & 18 & 19 & 20 & 21 & 22 \\[0.1cm]\hline\hline
    \mathrm{Max\: deg.\:} d_m & {66} & {61} & {64} & {63} & {62} & {63} & {64} & {63} & {62} & {65} & {64} & {63} & {62} & {63} & {64} & {63} & {62} & {63} & {64} & {63} & {62} & {63} & {62} \\[0.1cm]\hline\hline
  \end{array}
\end{displaymath}
\end{lem}

\begin{proof}
For order $1$ covariants, we know from Theorem~\ref{thm:cohen-macaulayness} that $\mathcal{M}=\mathbf{Cov}_{1} (\Sn{9})$ is a $\inv{\Sn{9}}$ Cohen--Macaulay module. We also obtain by a direct computation $\mathcal{H}_{\mathcal{M}}(z)=\displaystyle{{a(z)}/{p(z)}}$ with
\begin{align*}
a(z)&={z}^{5}+4\,{z}^{7}+10\,{z}^{9}+\cdots+{z}^{61}\,, \\
p(z)&=(1-z^{4})(1-z^{8})(1-z^{10})(1-z^{12})^2(1-z^{14})(1-z^{16})
\end{align*}
where the numerator $p(z)$ corresponds to a h.s.o.p of $\inv{\Sn{9}}$ given by
Proposition~\ref{prop:hsop_Inv9}. We deduce that the maximum degree of
a generator is $d_1=61$. Similar calculations yield the results for the other
orders $m$ (note that for invariants, which are order $0$ covariant, we make use of the $\inv{V}$ Hironaka decomposition \eqref{formule:Hironaka_Dec}).
\end{proof}

Similarly, we have this table for $\Sn{10}$.
\begin{lem}\label{lem:degree_Bound_S10}
The $\inv{\Sn{10}}$--module of m-th order covariant $\mathbf{Cov}_{m} (\Sn{10})$ is generated by covariants of maximum degree $d_m$ given in the following table, for $m\leqslant 26$:
    \begin{displaymath}\setlength{\arraycolsep}{2pt} 
        \hspace*{-1cm}\begin{array}{|c||c|c|c|c|c|c|c|c|c|c|c|c|c|c|}\hline
          \mathrm{Ord.}\ m & 0 & 2 & 4 & 6 & 8 & 10 & 12 & 14 & 16 & 18 & 20 & 22 & 24 & 26 \\[0.1cm]\hline\hline
          \mathrm{Max\: deg.\:} d_m & 59 & 45 & 46 & 45 & 46 & 47 & 46 & 45 & 46 & 45 & 46 & 45 & 45 & 45 \\\hline\hline
        \end{array}
    \end{displaymath}
\end{lem}


\section{Gordan's algorithm}\label{sec:Gordan_Alg}

Gordan's algorithm enables to compute a covariant basis for $\Sn{n}$, provided
that a covariant basis is known for $\Sn{m}$, $m < n$. Roughly speaking, it
consists in about $n/2$ iterations, each one giving a linear Diophantine
system to solve. We put emphasis on the computational aspects of this method in
this section. For more details, we refer the interested reader to the 19th century
literature~\cite{Gor1868,GY2010}, or to more modern works on that
topic~\cite{Wey1993,Croe2002,Olive2014a}.

\subsection{Relatively complete family and Gordan's ideal}

For a finite family of covariants (not necessarily a basis)
\begin{math}
  \mathrm{A}=\lbrace \ff_{1},\cdots,\ff_{p}\rbrace\subset \cov{\Sn{n}}\,,
\end{math}
we define $\cov{\mathrm{A}}$ to be the closure of $\mathrm{A}$ under transvectant operations,
\begin{displaymath}
  \bh_1,\bh_2\in \cov{\mathrm{A}} \Longrightarrow \tr{\bh_1}{\bh_2}{r}\in \cov{\mathrm{A}},\quad \forall r\in \mathbb{N}\,.
\end{displaymath}

\begin{defn}
Let $I\subset \cov{\Sn{n}}$ be a homogeneous ideal, a family $\mathrm{A}=\lbrace \ff_{1},\cdots,\ff_{p}\rbrace\subset \cov{\Sn{n}}$ of homogeneous covariants is \emph{relatively complete modulo $I$} if every homogeneous covariant $\bh\in \cov{\mathrm{A}}$ of degree $d$ can be written
\begin{displaymath}
	\bh=\mathbf{p}(\ff_{1},\dotsc,\ff_{p})+\bh_I \text{ with } \bh_I\in I,
\end{displaymath}
where $\mathbf{p}(\ff_{1},\dotsc,\ff_{p})$ and $\bh_I$ are degree $d$ homogeneous covariants.
\end{defn}

\begin{rem}
The notion of \emph{relatively complete} family is weaker than the one of \emph{generator set}. For instance, take $\uu\in \Sn{3}$ and 
\begin{displaymath}
	\bh_{2,2}:=\tr{\uu}{\uu}{2}\in \Sn{2},\quad
        \bh_{3,3}:=\tr{\uu}{\bh_{2,2}}{1}\in \Sn{3},\quad \Delta:=\tr{\bh_{2,2}}{\bh_{2,2}}{2}\,. 
\end{displaymath}
The family $\mathrm{A}_1=\lbrace \uu,\bh_{2,2},\bh_{3,3},\Delta\rbrace$ is a covariant basis of $\cov{\mathrm{A}_1}=\cov{\Sn{3}}$ and is thus a relatively complete family modulo $I=\set{0}$. Now, let
\begin{math}
	\mathrm{A}_2:=\set{\bh_{2,2},\Delta}.
\end{math}
We have $\cov{\mathrm{A}_2}\subsetneq \cov{\Sn{3}}$. Since $\mathrm{A}_2$ is exactly the covariant basis of the quadratic form $\bh_{2,2}\in \Sn{2}$, $\mathrm{A}_2$ is a relatively complete family modulo $I=\set{0}$ but is not a covariant basis of $\cov{\Sn{3}}$.
\end{rem}

\begin{defn}[Gordan's ideals]\label{def:Gord_Ideal}
  Let $r$ be an integer. We define the Gordan ideal $I_r$ to be the
  homogeneous ideal generated by the set of transvectants
  \begin{displaymath}
    \lbrace \tr{\bh}{\tr{\ff}{\ff}{r_1}}{r_2}\,:\, \bh\in \mathbf{Cov}_{d,m}(\Sn{n}),\quad  d,m\geq 0,\quad r_1\geqslant r,\quad r,\, r_1,\, r_2\in \mathbb{N} \rbrace\,.
  \end{displaymath}
\end{defn}
The ideal $I_r$ is clearly a homogeneous ideal, as being generated by homogeneous
elements. Moreover, we observe that:
\begin{itemize}
\item $I_r=\lbrace 0\rbrace$ for all $r>n$;
\item $I_{r+1}\subset I_r$ for all $r$;
\item $I_{2k-1}=I_{2k}$ for all $k\leqslant {n}/{2}$. 
\end{itemize}

\subsection{Linear Diophantine system}

Take
\begin{math}
	\mathrm{A}:=\set{ \ff_{1},\cdots,\ff_{p}},\ \mathrm{B}:=\set{ \bg_{1},\cdots,\bg_{q}}
\end{math}
to be two finite covariant families of $\cov{\Sn{n}}$ and consider the (infinite) family of transvectants
\begin{displaymath}
	\tr{\cU}{\cV}{r},\quad 
\text{ with }\quad 
	\cU:=\ff_1^{\alpha_1}\dotsc\ff_p^{\alpha_p},\quad
        \cV:=\bg_1^{\beta_1}\dotsc\bg_q^{\beta_q},\quad \alpha_i,\beta_j\in
        \mathbb{N}\,.
\end{displaymath}
Define $a_i$ (resp. $b_j$) to be the order of the covariant $\ff_i$ (resp. $\bg_j$). Now, to each non--vanishing transvectant
\begin{math}
	\tr{\cU}{\cV}{r}\,,
\end{math}
we can associate an integer solution $\boldsymbol{\kappa}:=((\alpha_i),(\beta_i),u,v,r)$ of the linear Diophantine system
\begin{equation}\label{Syst:Gordan_Linear}
{\mathcal S}(\mathrm{A},\mathrm{B}):\quad \begin{cases}
	a_{1}\alpha_{1} + \dotsc + a_{p}\alpha_{p} &= u+r\,, \\
	b_{1}\beta_{1} + \dotsc + b_{q}\beta_{q} &= v+r\,.
	\end{cases}
\end{equation}
An integer solution $\boldsymbol{\kappa}$ of ${\mathcal
  S}(\mathrm{A},\mathrm{B})$ is \emph{reducible} if we can decompose
$\boldsymbol{\kappa}$ as a sum of non--trivial solutions. Conversely, there
exists a finite family of \emph{irreducible integer solutions} of the system
${\mathcal S}(\mathrm{A},\mathrm{B})$ (see~\cite{Sta2012,Sta1983,Stu2008} for
details on linear Diophantine systems).\medskip

Now, to each integer solution $\boldsymbol{\kappa}$ of ${\mathcal
  S}(\mathrm{A},\mathrm{B})$, we can associate a well defined transvectant
$\tr{\cU}{\cV}{r}$. 
Define $\boldsymbol{\kappa}^1,\dotsc,\boldsymbol{\kappa}^l$ to be the irreducible integer solutions of ${\mathcal S}(\mathrm{A},\mathrm{B})$  and $\boldsymbol{\tau}^i$ to be their associated transvectants. Let $\ff\in \Sn{n}$, $\Delta\in \cov{\Sn{n}}$ be an invariant, $k\geqslant 0$ be a given integer and
\begin{math}
	\cH_{2k}:=\tr{\ff}{\ff}{2k}.
\end{math}
Finally, let $J_{2k+2}$ be either $I_{2k+2}$, or $I_{2k+2}+\langle \Delta \rangle$.\medskip

We have this important result.
\begin{thm}[\cite{GY2010}]\label{thm:FamRelComplete}
Suppose that $\mathrm{A}$ is relatively complete modulo $I_{2k}$ and contains the binary form $\ff$. Suppose also that $\mathrm{B}$ is relatively complete modulo $J_{2k+2}$ and contains the covariant $\cH_{2k}$. Then the family $\mathrm{C}:=\lbrace \boldsymbol{\tau}^1,\dotsc,\boldsymbol{\tau}^l\rbrace$ is relatively complete modulo $J_{2k+2}$ and
\begin{displaymath}
	\cov{\mathrm{C}}=\cov{\mathrm{A}\cup \mathrm{B}}=\cov{\Sn{n}}.
\end{displaymath}
\end{thm}

\subsection{The algorithm}\label{sec:The_Alg_of_Gordan}
On input a degree $n$, Gordan's algorithm returns a basis for the covariant
algebra $\cov{\Sn{n}}$ . All the details can be found
in~\cite{GY2010,Olive2014}.\medskip

First define $\ff\in \Sn{n}$ to be a single binary form and 
\begin{math}
	\cH_{2k}:=\tr{\ff}{\ff}{2k}.
\end{math}
The family $\mathrm{A}_{0}:=\set{ \ff}$ is relatively complete modulo
$I_{2}$. This means that every covariant $\bh\in \cov{\mathrm{A}_{0}}$ $(=\cov{\Sn{n}})$ can be written as
\begin{math}
\bh=\mathbf{p}(\ff)+\bh_2 \text{ with } \bh_2\in I_2. 
\end{math}
\medskip

Take now the covariant $\cH_{2}:=\tr{\ff}{\ff}{2}$ of order $2n-4$.
\begin{itemize}
\item If $2n-4>n$, we take $B_{0}:=\lbrace \cH_{2}\rbrace$ which is relatively complete modulo $J_4:=I_{4}$. Applying Theorem~\ref{thm:FamRelComplete} leads us to a family $\mathrm{A}_{1}:=\mathrm{C}$ relatively complete modulo $I_{4}$.\smallskip
\item If $2n-4=n$, we take $B_{0}:=\lbrace \cH_{2},\Delta\rbrace$ which is
  relatively complete modulo $J_4 := I_{4}+\langle \Delta \rangle$ with 
\begin{math}
	\Delta:=\tr{\tr{\ff}{\ff}{\frac{n}{2}}}{\ff}{n}.
\end{math}
In that case, by applying Theorem~\ref{thm:FamRelComplete}, we can take $\mathrm{A}_{1}$ to be $\mathrm{C}\cup \lbrace \Delta \rbrace$. A direct induction on the degree of the covariant shows that $\mathrm{A}_{1}$ is relatively complete modulo $I_{4}$.\smallskip
\item If $2n-4<n$, we suppose already known a covariant basis of $\Sn{2n-4}$. We then take $\mathrm{B}_0$ to be this basis, which is finite and relatively complete modulo $J_4:=I_{4}$ (because relatively complete modulo $\set{0}$). We apply Theorem~\ref{thm:FamRelComplete} to obtain $\mathrm{A}_{1}:=\mathrm{C}$.
\end{itemize}
\medskip

Let now be given by induction a finite family $\mathrm{A}_{k-1}$ containing
$\ff$ and relatively complete modulo $I_{2k}$. We consider the covariant
$\cH_{2k}:=\tr{\ff}{\ff}{2k}$.
\begin{itemize}
\item If $\cH_{2k}$ is of order $m>n$, we take $\mathrm{B}_{k-1}:=\lbrace
  \cH_{2k}\rbrace$ which is relatively complete modulo $J_{2k+2}:=I_{2k+2}$. By
  Theorem~\ref{thm:FamRelComplete} we take
  $\mathrm{A}_{k}:=\mathrm{C}$.\smallskip
\item If $\cH_{2k}$ is of order $m=n$, we take $\mathrm{B}_{k-1}:=\lbrace \cH_{2k},\Delta\rbrace$ which is relatively complete modulo $J_{2k+2}:=I_{2k+2}+\langle \Delta \rangle$ with 
\begin{math}
	\Delta:=\tr{\tr{\ff}{\ff}{\frac{n}{2}}}{\ff}{n}.
\end{math}
In that case, by applying Theorem~\ref{thm:FamRelComplete}, we can take
$\mathrm{A}_{k}$ to be $\mathrm{C}\cup \lbrace \Delta \rbrace$. A direct
induction on the degree of the covariant shows that $\mathrm{A}_{k}$ is
relatively complete modulo $I_{2k+2}$.\smallskip
\item If $\cH_{2k}$ is of order $m<n$, we suppose already known a covariant
  basis of $\Sn{m}$. We then take $\mathrm{B}_{k-1}$ to be this basis, which
  is relatively complete modulo $J_{2k+2}:=I_{2k+2}$ (because relatively complete modulo
  $\set{0}$). We directly apply Theorem~\ref{thm:FamRelComplete} to obtain
  $\mathrm{A}_{k}:=\mathrm{C}$.
\end{itemize}
\medskip

Finally, we have for $k=\lfloor n/2 \rfloor$ two cases, depending on $n$'s parity.
\begin{itemize}
\item If $n=2q$, we know that the family $\mathrm{A}_{q-1}$ is
  relatively complete modulo $I_{2q}$. Furthermore the family $B_{q-1}$ only
  contains the invariant $\Delta_q:=\lbrace \ff,\ff\rbrace_{2q}$. Set
\begin{math}
\mathrm{A}_{q}:=\mathrm{A}_{q-1}\cup \lbrace \Delta_q \rbrace\,.
\end{math}
\smallskip
\item If $n=2q+1$, the family $\mathrm{B}_{q-1}$ contains the quadratic form $\cH_{2q}:=\lbrace \ff,\ff\rbrace_{2q}$. We then know that the family $\mathrm{B}_{q-1}$ is given by the covariant $\cH_{2q}$ and the invariant $\delta_{q}:=\tr{\cH_{2q}}{\cH_{2q}}{2}$. By Theorem~\ref{thm:FamRelComplete}, set $\mathrm{A}_{q}:=\mathrm{C}$.
\end{itemize}
In both cases, $\mathrm{A}_{q}$ is relatively complete modulo
$I_{2q+2}=\lbrace 0\rbrace$, it is thus a covariant basis.

\section{Improvements of Gordan's algorithm}
\label{sec:impr-gord-algor}

\subsection{Shortened about relatively complete families}

One important idea, that dates back to Gordan~\cite{Gor1868} and Von
Gall~\cite{vGal1874} calculations, is to bypass the linear Diophantine system
using relations between covariants and arguing directly modulo some Gordan
ideal.  This typically yields directly the reduced systems $\mathrm{A}_{1}$
and $\mathrm{A}_{2}$, without using Theorem~\ref{thm:FamRelComplete}. Remind
also that we always have $\mathrm{A}_{0}=\set{\ff}$, for $\ff\in
\Sn{n}$~\cite{GY2010}.

\begin{lem}[\cite{Gor1875}]\label{lem:System_0_1}
For every integer $n\geqslant 4$, we have
\begin{displaymath}
	\mathrm{A}_{1}=\set{ \ff,\cH,\mathsf{T}},\quad \cH:=\tr{\ff}{\ff}{2},\quad \mathsf{T}:=\tr{\ff}{\cH}{1}.
\end{displaymath}
For every integer $n\geqslant 8$, we have
\begin{displaymath}
	\mathrm{A}_{2}=\set{ \ff,\cH,\mathsf{T},\mathsf{K},\tr{\ff}{\mathsf{K}}{1},\tr{\ff}{\mathsf{K}}{2},\tr{\cH}{\mathsf{K}}{1}},\quad \mathsf{K}:=\tr{\ff}{\ff}{4}.
\end{displaymath}
\end{lem}

\subsection{Injective companion of a linear Diophantine system}
\label{sec:inject-comp-line}

We generalize to our situation the approach proposed by Clausen and
Fortenbacher in the case of one equation~\cite{CF1990}, based on what they
called the \emph{injective companion} of a linear Diophantine equation.\medskip

Start from a system composed of two equations, written as
\begin{equation}\label{eq:Linear_Syst_1}
\begin{cases}
	\sum_{i\in I} a_{i}\left(\sum_{l=1}^{s(i)} \alpha_{il}\right) &=u+r\,, \\
	\sum_{j\in J} b_{j}\left(\sum_{m=1}^{t(j)} \beta_{jm}\right) &=v+r\,,
\end{cases}
\end{equation}
with finite sets $I,J$ of positive integers, mappings $s:\,I\rightarrow
\mathbb{N}^{*}$, $t:\,J\rightarrow \mathbb{N}^{*}$ and natural integers
$(a_{i})$, $(b_{j})$, $(\alpha_{il})$, $(\beta_{jm})$, $u$, $v$ and $r$. We
now consider its injective companion
\begin{equation}\label{eq:Linear_Syst_2}
\begin{cases}
	\sum_{i\in I} a_{i}\alpha_{i} &=u+r\,, \\
	\sum_{j\in J} b_{j}\beta_{j} &=v+r\,.
\end{cases}
\end{equation}

With a proof which is essentially the same as in~\cite{CF1990}, we obtain the
next result.
\begin{lem}
  $((\alpha_{il}),(\beta_{jm}),u,v,r)$ is a (minimal) solution of the linear
  Diophantine system \eqref{eq:Linear_Syst_1} if and only if
  $((\alpha_{i}),(\beta_{j}),u,v,r)$ is a (minimal) solution of the injective
  companion \eqref{eq:Linear_Syst_2}, where
  \begin{displaymath}
    \alpha_{i}:=\sum_{l=1}^{s(i)} \alpha_{il}\text{ and }\beta_{j}:=\sum_{m=1}^{t(j)} \beta_{jm}.
  \end{displaymath}	
\end{lem}

\begin{rem} \label{rem:inject-comp-line-1}
  Given some $\alpha_i \geqslant 0$, the number of solutions $(\alpha_{il})
  \geqslant 0$ of
  \begin{math}
    {\alpha_i = {\alpha}_{i1}+ {\alpha}_{i2} + \ldots + {\alpha}_{is(l)}}
  \end{math}
  is equal to the binomial coefficient ${\binom {\alpha_i}{\alpha_i+s(\ell)-1}}$.
\end{rem}

\subsection{Relations on weighted monomials}
\label{sec:relat-weight-monon}
Our aim is to take advantage of relations between covariants to ease some of
the calculations in Gordan's algorithm.  Note that the proofs of the results
given below can be found in~\cite{Olive2014}.

\subsubsection{Commutative algebra}
\label{sec:general-result}

Let 
\begin{math}
  x_1>x_2>\dotsc >x_p
\end{math}
be indeterminates and $\mathcal{A}=\GC[x_1,\dotsc,x_p]$ be a graded algebra of finite type. Consider also the lexicographic order on monomials of $\mathcal{A}$. We write $\mathbf{m}_1\mid \mathbf{m}_2$ whenever the monomial $\mathbf{m}_1$ divides the monomial $\mathbf{m}_2$.\medskip

Now, assume that there exist relations of those two different types.

\begin{hypo}\label{hypo:Rel_Mon_1facteur}
There exists a finite family $I\subset \set{1,\dotsc,p-1}$ and for each $i\in I$ a relation
\begin{equation}\label{Rel:Ri}
  (\mathcal{R}_i),\quad x_i^{a_i}=\sum_{k=0}^{a_i-1} x_i^k\mathbf{p}_k(x_{i+1},\dotsc,x_p),\quad a_i\in \mathbb{N}^* 
\end{equation}
where $\mathbf{p}_k$ is some polynomial. We write $\mathbf{m}_i:=x_i^{a_i}$.
\end{hypo}

\begin{hypo}\label{hypo:Rel_Mon_2facteur}
There exists a finite family $J$ and for each $j\in J$ a relation
\begin{equation}\label{Rel:R'j}
  (\mathcal{R}'_j),\quad x_{j_b}^{b_{j_b}}x_{j_c}^{c_{j_c}}=\mathbf{p}(x_{j_c+1},\dotsc,x_p),\quad b_{j_b},c_{j_c}\in \mathbb{N}^* 
\end{equation}
where $x_{j_b}>x_{j_c}$ and $\mathbf{p}$ is some polynomial. We write $\mathbf{m}'_{j}:=x_{j_b}^{b_{j_b}}x_{j_c}^{c_{j_c}}$.
\end{hypo}

\begin{lem}\label{lem:Reduce_By_Weighted_Relations}
  Under Hypothesis~\ref{hypo:Rel_Mon_1facteur} and
  Hypothesis~\ref{hypo:Rel_Mon_2facteur}, the algebra $\mathcal{A}$ is
  generated by the family of monomials $\mathbf{m}$ such that
\begin{displaymath}
  \mathbf{m}_i\nmid \mathbf{m},\quad \mathbf{m}'_{j}\nmid \mathbf{m},\quad \forall i\in I,\quad \forall j\in J\,.
\end{displaymath}
\end{lem}

\subsubsection{Application to Gordan's algorithm}
Gordan's algorithm deals with families $\mathrm{A}_{0},$ $\mathrm{B}_{0},$
$\dotsc$ (see \autoref{sec:The_Alg_of_Gordan}). Consider the case where the
family $\mathrm{B}_{k-1}$ is the covariant basis of the binary form
\begin{displaymath}
	\cH_{2k}=\tr{\ff}{\ff}{2k}\in \Sn{2n-4k}.
\end{displaymath}
$\cH_{2k}$ is of order $2n-4k<n$ and we suppose known its covariant basis. As
in Theorem~\ref{thm:FamRelComplete}, write $\Delta\in \cov{\Sn{n}}$ to be an
invariant and $J_{2k+2}:=I_{2k+2}$ or $J_{2k+2}:=I_{2k+2}+\langle \Delta
\rangle$. Write $\mathrm{A}:=\mathrm{A}_{k-1}$, $\mathrm{B}:=\mathrm{B}_{k-1}$
and note $\mathrm{C}$ to be the finite family of transvectants
\begin{math}
	\tr{\cU}{\cV}{r}
\end{math}
associated to irreducible solutions
of the Diophantine system
${\mathcal S}(\mathrm{A},\mathrm{B})$ (\textit{cf. }Equation \eqref{Syst:Gordan_Linear}).
Finally, suppose that Hypothesis~\ref{hypo:Rel_Mon_1facteur} and
Hypothesis~\ref{hypo:Rel_Mon_2facteur} hold for the basis
$\mathrm{B}$ of the algebra $\cov{\Sn{2n-4k}}$.

\begin{thm}\label{thm:FamRelCompleteBis}
With the notations of Theorem~\ref{thm:FamRelComplete}, the subfamily $\tilde{\mathrm{C}}$ of $\mathrm{C}$ given by 
\begin{displaymath}
	\tr{\cU}{\tilde{\cV}}{r}\in \mathrm{C},\quad \mathbf{m}_i\nmid \tilde{\cV},\quad \mathbf{m}'_{j}\nmid 
	\tilde{\cV},\quad \forall i\in I,\quad \forall j\in J
\end{displaymath}
is relatively complete modulo $J_{2k+2}$ and
\begin{displaymath}
	\cov{\tilde{\mathrm{C}}}=\cov{\mathrm{A}\cup \mathrm{B}}=\cov{\Sn{n}}.
\end{displaymath}
\end{thm}


\section{Computational aspects}
\label{sec:computation}

\subsection{Reformulating Theorem~\ref{thm:FamRelComplete}}
\label{sec:reform-theor-refthm}

The most computationally intensive steps of Gordan's algorithm are the ones
which make use of Theorem~\ref{thm:FamRelComplete} in order to obtain the
families $\mathrm{A}_{k}$. If Lemma~\ref{lem:System_0_1} yields
$\mathrm{A}_{1}$, $\mathrm{A}_{2}$ when $n\geqslant 8$ and if there exists a
similar result for $\mathrm{A}_{k}$ when $2k \leq {n}/{2}$ (or equivalently
when the order of $\tr{\ff}{\ff}{2k}$ is greater or equal to $n$), we have in
the remaining cases to solve a linear Diophantine system.

It turns out that using its injective companion as explained in
Section~\ref{sec:inject-comp-line} enables to find its minimal solutions, even
if their number is very large (at least in degree 9 and 10, see
Section~\ref{sec:results} for details).  Now, the covariants
$\boldsymbol{\tau}^l$ of family $\mathrm{C}$ in
Theorem~\ref{thm:FamRelComplete} associated to most of these solutions
have large degrees and orders. Writing them as a polynomial is simply
hopeless.

We solve this first issue as in~\cite{LR2012}. In Gordan's algorithm,
covariants result from transvectants of products of covariants, each one also
recursively defined by transvectants. We thus represent them by the sequence
of transvectants that must be done to obtain them. We do not have anymore their
polynomial expressions, but we can still evaluate them on a binary form.  In
other words, a covariant is represented by an evaluation program. Note that it
is immediate to determine the degree and the order of a covariant from the
sequence of operations coded in an evaluation program.\smallskip

Another difficulty is that taking for the family $\mathrm{A}_k$ all the
corresponding transvectants $\boldsymbol{\tau}^l$'s yield huge computations in the
following steps of Gordan's algorithm. To avoid this, we substitute to the
family $\mathrm{A}_{k}$ a family $\mathrm{A}'_{k}$ which spans
$\mathrm{A}_{k}$.  The purpose is to have $\mathrm{A}'_{k}$ much smaller than
$\mathrm{A}_{k}$ (typically, few hundred of covariants instead of billions in
our cases of interest). Incidentally, $\mathrm{A}'_{k}$ contains the binary
form $\ff\in \Sn{n}$ and is still relatively complete modulo $I_{2k+2}$.

\begin{figure}
  \centering
  \fbox{\parbox{.8\textwidth}{\small\begin{enumerate}
    \item $\gset_1 \leftarrow \lbrace \ff \rbrace$
    \item For $d\ =\ 2,\ldots, d_{max}$ :\medskip
      \begin{enumerate}
      \item $\gset_d \leftarrow \lbrace\rbrace$\smallskip
      \item For each $\cH = \prod_{\bh \in \gset_i} \bh$\, s.t. $\deg \cH
        = d$~:\smallskip
        \begin{itemize}
        \item If $\cH \notin \langle \gset_d \rangle$ then {$\gset_d \leftarrow \gset_d \cup \lbrace
            \cH \rbrace$}\\ \medskip
            where $\langle \gset_d \rangle$ is the algebra generated by $\gset_d$.
        \end{itemize}
      \item $\displaystyle \Pi_{d-1} = \lbrace \cH\ \ |\ \ \cH = \hspace*{-0.8cm}\prod_{
          \begin{footnotesize}
            \begin{array}{l}
              \bh \in \gset_i,\ \ord \bh \neq 0,\ \deg{\bh} \geqslant 2
            \end{array}
          \end{footnotesize}
      } \hspace*{-1.5cm} \bh\ \hspace*{1cm}\textrm{ and }\deg \cH = {d-1}\rbrace$\medskip
    \item For each $\mathbf{F} \in \Pi_{d-1}$, and each possible level $r$~:\smallskip
      \begin{itemize}
      \item If $\tr{\mathbf{F}}{\ff}{r} \notin \langle \gset_d \rangle$ then {$\gset_d \leftarrow \gset_d \cup \lbrace\tr{\mathbf{F}}{\ff}{r}\rbrace$}
      \end{itemize}
    \end{enumerate}
  \end{enumerate}}}
\caption{Olver's algorithm}
\label{fig:olver}
\end{figure}

To define the family $\mathrm{A}'_{k}$, we start from an
algorithm\footnote{Note that Olver's algorithm has only a \emph{running bound}
  as shown by~\cite{BRT2006}. } published by Olver~\cite[p. 144]{Olv1999} that
aims at computing a basis for the sub-algebra of $\cov{\Sn{n}}$ defined
by generators of degree upper-bounded by some constant $d_{\max}$. Usually, as a
preamble to Gordan's algorithm, it is good practice to run this algorithm for
some $d_{\max}$ chosen large enough to obtain a good candidate minimal basis
$\gset = \cup_{d \leqslant d_{\max}} \gset_{d}$ for $\cov{\Sn{n}}$
(\textit{cf.}  Algorithm~\ref{fig:olver}). 

This done, we consider in turn all the couples $(d,m)$ of degrees/orders in
the family $\mathrm{A}_{k}$, sorted as considering first the spaces
$\mathbf{Cov}_{d,m}(\Sn{n})$ of smallest dimension. For each $(d,m)$, we check
using linear algebra, \textit{e.g.}  Algorithm~\ref{fig:linearalgebra}
(\textit{cf.}  Section~\ref{sec:linear-algebra}), that the dimension of the
homogeneous space $\langle \gset \rangle_{d,m}$ is exactly the one of
$\mathbf{Cov}_{d,m}(\Sn{n})$. The latter is given by Springer formula
(\textit{cf.}  Theorem~\ref{thm:springer}). So, we ensure that
$\langle\mathrm{A}_{k}\rangle_{d,m}\subset \mathbf{Cov}_{d,m}(\Sn{n}) =
\langle \gset \rangle_{d,m}$ for all the couples $(d,m)$, and thus that
$\mathrm{A}_{k}\subset \langle \gset \rangle$. We can then define $\mathrm{A}'_k$ as
the subset of $\gset$ that spans $\mathrm{A}_k$, or more precisely,
\begin{equation}\label{eq:3}
  \mathrm{A}'_k:=\lbrace\cc \in \gset\,|\, \exists \ \boldsymbol{\tau} \in \mathrm{A}_k\  \mathrm{with}\ \deg \cc \leqslant
  \deg \boldsymbol{\tau}\ \mathrm{and}\  \ord \cc \leqslant \ord \boldsymbol{\tau}\rbrace\,.
\end{equation}
Under this viewpoint, we may see Gordan's algorithm as a way of having upper-bounds
to prove that the basis returned by Olver's algorithm is minimal.\medskip

We did not encounter the problem in our $\cov{\Sn{9}}$ and $\cov{\Sn{10}}$
calculations, but it might be possible that Algorithm~\ref{fig:linearalgebra}
does not terminate at all for some $(d,m)$. This could be either because the
basis $\gset$ is incomplete, or simply because of unfortunate random draws in
the algorithm. To avoid this, let us define the subset $(\mathrm{A}_k)_{d,m}$
to be the degree $d$ and order $m$ covariants of $\mathrm{A}_k$. We then
suggest to stop Algorithm~\ref{fig:linearalgebra} after a timeout and then
check if there exist transvectants $\boldsymbol{\tau}$ in
$(\mathrm{A}_k)_{d,m}$ that can complete the basis of covariants constructed
so far. We may perform this task as in Step (4) of this algorithm where we
replace the covariant random draws at Step (4.a) by the enumeration of
$(\mathrm{A}_k)_{d,m}$.  Of course, we have to enlarge the set defined in
Equation~\eqref{eq:3} with these $\boldsymbol{\tau}$ to define here
$\mathrm{A}'_k$.

Still, we stress that we are in trouble when, despite all that,
$\langle\mathrm{A'}_k\rangle_{d,m} \subsetneq \mathbf{Cov}_{d,m}(\Sn{n})$,
since we can not exclude that, again due to unfortunate random draws, this
procedure wrongly detects that some evaluation of $\boldsymbol{\tau}$ is in
the associated projection of $\langle\gset\rangle$, while the covariant
$\boldsymbol{\tau}$ itself is not in $\langle\gset\rangle$.  Missing such a
$\boldsymbol{\tau}$ might yield at the end of Gordan's algorithm a wrong basis
for $\cov{\Sn{n}}$\,. In such a very exceptional case, the best in our opinion
is to restart from the beginning the whole computation with a better basis
$\langle\gset\rangle$, which means running Olver's algorithm with a largest
$d_{\mathrm{max}}$\,.

\medskip

Now, we can optimize all the computations using different techniques: a first
one based on upper-bounds on degrees and orders, the second one based on
computation reduction and the third one based on linear algebra.

\subsection{Upper-bounds on degrees and orders}
\label{sec:bounds-degr-orders}

Now, several of the improvements stated in the paper come into play.  We can
first reduce some covariants of the family $\mathrm{A}_{k}$ using the relations that we may have
calculated between covariants of $\mathrm{B}_{k-1}$ (see
Theorem~\ref{thm:FamRelCompleteBis}). Typically, assuming that we have ordered
the covariants of $\mathrm{B}_{k-1}$ by some inequality relation $<$, if we have for some $\mathsf{C}_1 >
\mathsf{C}_2$ a relation of the form
\begin{displaymath}
  \mathsf{C}_1^{e_1}\times \mathsf{C}_2^{e_2}  = \sum \prod_{\mathsf{C} < \mathsf{C}_2} \mathsf{C}\,
\end{displaymath}
then we have Hypothesis~\ref{hypo:Rel_Mon_2facteur} and we can use Theorem~\ref{thm:FamRelCompleteBis}. Thus we can discard minimal solutions that yield transvectants of the form
\begin{displaymath}
  {\left( \prod_{\cc_i\in \mathrm{A}_{k-1}} \cc_i^{a_i}, \ {\mathsf{C}_1^{e_1}\times \mathsf{C}_2^{e_2}}
      \prod_{\mathfrak{u}_i\in \mathrm{B}_{k-1}} \mathfrak{u}_i^{b_i} \right)_r}\,.
\end{displaymath}
This process often results in decreasing the number of couples $(d,m)$ of
degrees/orders in the family $\mathrm{A}_{k}$.\medskip

We can further make use of upper-bounds known on the degree $d$ and the order
$m$ of a basis. Especially, we know
\begin{itemize}
\item from Grace and Young~\cite{GY2010}, that the order $m$ can be upper-bounded by
  Equation~(\ref{eq:1}),
\item from Van Den Bergh~\cite{MR1128219,Van1994} and the Cohen--Macaulayness property
  of some $\mathbf{Cov}_{m}(\Sn{n})$ module (see
  Theorem~\ref{thm:cohen-macaulayness}), that for orders $m$ of medium size,
  the degree $d$ can be upper-bounded by the largest exponent that arises in
  the numerator of a rational expression of the Hilbert series of
  $\mathbf{Cov}_{m}(\Sn{n})$, related to a given h.s.o.p of $\inv{\Sn{n}}$ .
\end{itemize}
The later upper-bounds are very spectacular in practice, see
Lemma~\ref{lem:degree_Bound_S9} or Lemma~\ref{lem:degree_Bound_S10}.\medskip

\subsection{Reductions}
\label{sec:reductions}

We may remark that the computations can be done
\begin{itemize}
\item modulo a subfamily $\theta_1$, \ldots, $\theta_j$ of a system of
  parameters for $\inv{S_n}$,
\item modulo a small prime $p$, typically $p=65521$.
\end{itemize}
In both cases, if the images of the covariants under reduction are
independent, then the covariants are independent. So, instead of checking that
the dimension of some $\mathcal{M}=\mathbf{Cov}_{d,m}(\Sn{n})$ satisfies Springer's dimension
over $\mathbb Q$ or $\mathbb C$, it is enough to check that the dimension of
${\mathcal M}/(\theta_1{\mathcal M} + \ldots + \theta_j{\mathcal M} )$ is modulo $p$ the one derived from the
Hilbert series given by Equation~(\ref{eq:Hilbert_S_Quotient}).
  
\subsection{Linear algebra}
\label{sec:linear-algebra}

\begin{figure}
  \centering
  \fbox{\parbox{.9\textwidth}{\small \begin{enumerate}
      \item Draw $(1+O(1))\times \dim \mathbf{Cov}_{d,m}(\Sn{n})$ random covariants $(\cc_i)$ in
        $\langle\gset\rangle_{d,m}$\,.\medskip
      \item Evaluate these $\cc_i$ at $\dim \mathbf{Cov}_{d,m}(\Sn{n})/(k+1)+O(1)$ forms $(\ff_j)$
        chosen at random\\\centerline{this yields a matrix $\mathtt M = (\,\cc_i(\ff_j)\,)_{i,j}$\,.}\medskip
      \item Compute the ``parity-check'' matrix ${\mathtt M}^{\top}$ of
        $\mathtt M$.
        \begin{itemize}
        \item[]
          \begin{footnotesize}
            (\textit{i.e} ${\bf c}\times {\mathtt M}^{\top} =
            \bf 0$ if ${\bf c} = {\mathtt M} \times {\bf v}$ for some vector ${\bf v}$)
          \end{footnotesize}
        \end{itemize}\medskip
      \item While $\rank \mathtt M\, < \dim \mathbf{Cov}_{d,m}(\Sn{n})$,\smallskip
        \begin{enumerate}
        \item look for a  ${\cc}\in \langle\gset\rangle_{d,m}$ s.t. ${\mathtt M}^{\top}
          \times  (\,{\cc}(\ff_j)\,)\neq 0$,\smallskip
        \item update ${\mathtt M}^{\top}$\,.
        \end{enumerate}
      \end{enumerate}
    }}
  \caption{Checking that a covariant family $\gset$  generates $\mathbf{Cov}_{d,m}(\Sn{n})$}%
\label{fig:linearalgebra}
  \end{figure}

  To check the dimensions of the numerous homogeneous spaces $\langle \gset
  \rangle_{d,m}$ that arise in Gordan's algorithm, we finally proceed as in
  the Las-Vegas type probabilistic Algorithm~\ref{fig:linearalgebra}. It
  terminates with the correct answer, but it's running time depends on random
  choices.  Its main advantage is that it makes only one Gauss elimination.

  It works as follows:
  \begin{itemize}
  \item in Step (1), we first chose at random slightly more
    than $D:=\dim \mathbf{Cov}_{d,m}(\Sn{n})$ covariants;
  \item we evaluate them in Step (2) at enough binary forms to obtain a matrix
    $\mathtt{M}$ (remember that these covariants are given by an evaluation
    program);
  \item we compute in Step (3) by a Gauss elimination a basis for the dual of the
    vector space defined by $\mathtt{M}$, and incidentally we have the rank of
    $\mathtt{M}$ (it is expected that this rank is close to $D$);
  \item in the Step (4) loop, we look for a covariant the evaluation vector of
    which is not orthogonal to $\mathtt{M}^{\top}$ and when this is the case,
    update ${\mathtt M}^{\top}$ (this can be done incrementally with complexity
    only $O(D^2)$).
  \end{itemize}
  
  In such a computation, we observe that it is relatively easy to find
  generators that span a subspace of $\mathbf{Cov}_{d,m}(\Sn{n})$ with a dimension
  close to the awaited dimension $D$. Actually, most of the time is finally
  spend in Step (4.a) while looking for the very few additional generators
  needed to reach the dimension $D$.  But this step is straightforward enough
  to be easily implemented and optimized in a program written at low level, in
  language \textsc{C}. Furthermore, it is highly parallelizable on a multi-core
  computer: give one covariant $\cc$ to each core, either while
  computing $\mathtt M$\,, or looking for a covariant $\cc$
  s.t. ${\mathtt M}^{\top} \times (\,{\cc}(\ff_j)\,)\neq 0$\,.

  It may be worth to add that here ``choosing at random covariants'' simply
  means to first choose at random products of covariants in $\gset$ with the
  expected degree and a large enough order, and then to calculate a
  transvectant of suitable level. We just observe in practice that avoiding to
  take invariants as terms in these products improves significantly the
  chances of finding a new independent covariant in Step (4.a).


\section{Results}
\label{sec:results}


\subsection{Covariant basis of binary nonics}
\label{sec:covar-basis-binary}

We start from the possibly incomplete basis known for the algebra
$\cov{\Sn{9}}$. Such a basis have been already computed in the past, (see
for instance~\cite{Brouwer2015}). We develop our own implementation of Olver's
algorithm and ran it with a large upper-bound $d_{\max}$, \textit{e.g.}
$d_{\max}=30$. We retrieve a basis with $476$ generators.  A complete but
somehow unappealing definition for these generators is in
Appendix~\ref{sec:mathbfcov_9}. We give instead in Table~\ref{tab:cov9} the
number of generators for each degree and each order.\medskip

All in all, the calculations that follow enables us to prove this theorem.
\begin{thm}\label{thm:covar-basis-binary-9}
  The 476 covariants given in Appendix~\ref{sec:mathbfcov_9} define a
  minimal basis for the covariant algebra $\cov{\Sn{9}}$\,.
\end{thm}

\begin{table}[htbp]
  \centering
  \begin{footnotesize}
    \begin{displaymath}\setlength{\arraycolsep}{2pt}\hspace*{-0.5cm} \begin{array}{c||cccccccccccccccccccccc|c|c|c}
        d.\!/o.\! &   0 &   1 &   2 &   3 &   4 &   5 &   6 &   7 &   8 &   9 &  10 &  11 &  12 &  13 &  14 &  15 &  16 &  17 &  18 &  19 &  21 &  22 &   \# & Cum \\\hline\hline
        1 &   - &   - &   - &   - &   - &   - &   - &   - &   - &   1 &   - &   - &   - &   - &   - &   - &   - &   - &   - &   - &   - &   - &   1 &   1 \\
        2 &   - &   - &   1 &   - &   - &   - &   1 &   - &   - &   - &   1 &   - &   - &   - &   1 &   - &   - &   - &   - &   - &   - &   - &   4 &   5 \\
        3 &   - &   - &   - &   1 &   - &   1 &   - &   1 &   - &   2 &   - &   1 &   - &   1 &   - &   1 &   - &   1 &   - &   - &   1 &   - &  10 &  15 \\
        4 &   2 &   - &   - &   - &   2 &   - &   2 &   - &   3 &   - &   2 &   - &   2 &   - &   2 &   - &   1 &   - &   1 &   - &   - &   1 &  18 &  33 \\
        5 &   - &   1 &   - &   3 &   - &   4 &   - &   4 &   - &   3 &   - &   4 &   - &   2 &   - &   3 &   - &   - &   - &   1 &   - &   - &  25 &  58 \\
        6 &   - &   - &   4 &   - &   4 &   - &   6 &   - &   6 &   - &   3 &   - &   4 &   - &   - &   - &   1 &   - &   - &   - &   - &   - &  28 &  86 \\
        7 &   - &   4 &   - &   7 &   - &   8 &   - &   7 &   - &   6 &   - &   1 &   - &   1 &   - &   - &   - &   - &   - &   - &   - &   - &  34 & 120 \\
        8 &   5 &   - &   8 &   - &  10 &   - &  10 &   - &   4 &   - &   2 &   - &   - &   - &   - &   - &   - &   - &   - &   - &   - &   - &  39 & 159 \\
        9 &   - &   9 &   - &  14 &   - &  10 &   - &   7 &   - &   1 &   - &   - &   - &   - &   - &   - &   - &   - &   - &   - &   - &   - &  41 & 200 \\
        10 &   5 &   - &  15 &   - &  15 &   - &   3 &   - &   1 &   - &   - &   - &   - &   - &   - &   - &   - &   - &   - &   - &   - &   - &  39 & 239 \\
        11 &   - &  17 &   - &  16 &   - &   7 &   - &   1 &   - &   - &   - &   - &   - &   - &   - &   - &   - &   - &   - &   - &   - &   - &  41 & 280 \\
        12 &  14 &   - &  23 &   - &   4 &   - &   1 &   - &   - &   - &   - &   - &   - &   - &   - &   - &   - &   - &   - &   - &   - &   - &  42 & 322 \\
        13 &   - &  25 &   - &  10 &   - &   1 &   - &   - &   - &   - &   - &   - &   - &   - &   - &   - &   - &   - &   - &   - &   - &   - &  36 & 358 \\
        14 &  17 &   - &  13 &   - &   1 &   - &   - &   - &   - &   - &   - &   - &   - &   - &   - &   - &   - &   - &   - &   - &   - &   - &  31 & 389 \\
        15 &   - &  26 &   - &   1 &   - &   - &   - &   - &   - &   - &   - &   - &   - &   - &   - &   - &   - &   - &   - &   - &   - &   - &  27 & 416 \\
        16 &  21 &   - &   3 &   - &   - &   - &   - &   - &   - &   - &   - &   - &   - &   - &   - &   - &   - &   - &   - &   - &   - &   - &  24 & 440 \\
        17 &   - &   7 &   - &   - &   - &   - &   - &   - &   - &   - &   - &   - &   - &   - &   - &   - &   - &   - &   - &   - &   - &   - &   7 & 447 \\
        18 &  25 &   - &   - &   - &   - &   - &   - &   - &   - &   - &   - &   - &   - &   - &   - &   - &   - &   - &   - &   - &   - &   - &  25 & 472 \\
        19 &   - &   1 &   - &   - &   - &   - &   - &   - &   - &   - &   - &   - &   - &   - &   - &   - &   - &   - &   - &   - &   - &   - &   1 & 473 \\
        20 &   2 &   - &   - &   - &   - &   - &   - &   - &   - &   - &   - &   - &   - &   - &   - &   - &   - &   - &   - &   - &   - &   - &   2 & 475 \\
        21 &   - &   - &   - &   - &   - &   - &   - &   - &   - &   - &   - &   - &   - &   - &   - &   - &   - &   - &   - &   - &   - &   - &   - & 475 \\
        22 &   1 &   - &   - &   - &   - &   - &   - &   - &   - &   - &   - &   - &   - &   - &   - &   - &   - &   - &   - &   - &   - &   - &   1 & 476 \\\hline
        Tot &  92 &  90 &  67 &  52 &  36 &  31 &  23 &  20 &  14 &  13 &   8 &   6 &   6 &   4 &   3 &   4 &   2 &   1 &   1 &   1 &   1 &   1 & 476 &     \\
      \end{array}
    \end{displaymath}
  \end{footnotesize}
  \caption{Minimal basis of $\cov{\Sn{9}}$}
  \label{tab:cov9}
\end{table}

\subsubsection{Gordan iterations}
\label{sec:gordans-iterations}

We follow~\autoref{sec:The_Alg_of_Gordan}.\medskip
\begin{itemize}
\item We start from {$\mathrm{A}_0 = \lbrace\ff\rbrace$}.
\item Set $\cH:=\tr{\ff}{\ff}{2}$\,, we know from Lemma~\ref{lem:System_0_1} that
  \begin{displaymath}
    {\mathrm{A}_1=\lbrace\ff,\cH,\ \mathsf{T}\rbrace}\text{ with }
    \mathsf{T}:=\tr{\ff}{\cH}{1}.
  \end{displaymath}
\item The family $\mathrm{A}_2$ is (similarly) given by
  \begin{footnotesize}
    \begin{displaymath}\setlength{\arraycolsep}{0.1cm}
      \hspace*{-1.4cm}\begin{array}{|l||c|c|c|c|c|c|c|}\hline
        Cov. &\hspace*{0.2cm}{\ff}\hspace*{0.2cm} &\bh_{10}:=\tr{\ff}{\ff}{4} &\bh_{14}:=\tr{\ff}{\ff}{2} &\tr{\ff}{\bh_{10}}{2} &\tr{\ff}{\bh_{10}}{1} &\tr{\ff}{\bh_{14}}{1} &\tr{\bh_{10}}{\bh_{14}}{1}\\\hline
        Ord. &9 &10 &14 &15 &17 &21 &22\\
        Deg. &1 &2 &2 &3 &3 &3 &4\\\hline
      \end{array}
    \end{displaymath}
    \end{footnotesize}
  \end{itemize}
  \medskip

  We have now to compute with Theorem~\ref{thm:FamRelComplete} the family
  $\mathrm{A}_{3}$, using the family $\mathrm{A}_{2}$\,.\medskip

  \hspace*{-1cm}\begin{tabular}{ll}
    \begin{minipage}{0.7\linewidth}
      \begin{itemize}
      \item Let $\mathrm{B}_{2}$ be the covariant basis of
        $\cc_3:=\tr{\ff}{\ff}{6}\in \Sn{6}$. As a classical
        result~\cite{GY2010}, such a basis is given by 26 covariants, and we
        only keep the family of $26-5=21$ covariants (see the next
        table).\medskip
      \item $\mathrm{A}_3$ is then composed of transvectants of the type
        \begin{equation}\label{eq:Trans_Cov_9}
          \left( \prod_{\bh\in \mathrm{A}_2} \bh^a, \prod_{\mathsf{C}\in \mathrm{B}_{2}} \mathsf{C}^b \right)_r
        \end{equation}
        We keep a finite number of them, those that come from minimal
        solutions of a linear Diophantine system.\linebreak This is the
        difficult part (see Section~\ref{sec:line-integ-syst}).\medskip
      \item We take few additional transvectants with $\cc_2=\tr{\ff}{\ff}{8}\in
        \Sn{2}$, and we finally obtain a covariant basis for $\cov{\Sn{9}}$.
      \end{itemize}
    \end{minipage}
    &
    \begin{scriptsize}
      \begin{math}\setlength{\arraycolsep}{1pt} \begin{array}{c||ccccccc|c|c|c}
          d/o & 0 & 2 & 4 & 6 & 8 & 10 & 12 & \# & Cum \\\hline\hline
          1 &   - &   - &   - &   1 &   - &   - &   - &   1 &   1 \\
          2 &   1 &   - &   1 &   - &   1 &   - &   - &   3 &   4 \\
          3 &   - &   1 &   - &   1 &   1 &   - &   1 &   4 &   8 \\
          4 &   1 &   - &   1 &   1 &   - &   1 &   - &   4 &  12 \\
          5 &   - &   1 &   1 &   - &   1 &   - &   - &   3 &  15 \\
          6 &   1 &   - &   - &   2 &   - &   - &   - &   3 &  18 \\
          7 &   - &   1 &   1 &   - &   - &   - &   - &   2 &  20 \\
          8 &   - &   1 &   - &   - &   - &   - &   - &   1 &  21 \\
          9 &   - &   - &   1 &   - &   - &   - &   - &   1 &  22 \\
          10 &   1 &   1 &   - &   - &   - &   - &   - &   2 &  24 \\
          11 &   - &   - &   - &   - &   - &   - &   - &   - &  24 \\
          12 &   - &   1 &   - &   - &   - &   - &   - &   1 &  25 \\
          13 &   - &   - &   - &   - &   - &   - &   - &   - &  25 \\
          14 &   - &   - &   - &   - &   - &   - &   - &   - &  25 \\
          15 & 1 & - & - & - & - & - & - & 1 & 26 \\\hline
          Tot &   5 &   6 &   5 &   5 &   3 &   1 &   1 &  26 &     \\
        \end{array}
      \end{math}
    \end{scriptsize}
\end{tabular}
\bigskip

\subsubsection{Linear integer system}
\label{sec:line-integ-syst}

Taking the orders of covariants occurring in transvectants given
by~\eqref{eq:Trans_Cov_9} lead to the following integer system ($(a_i)$, $(b_j)$,
$u$, $v$, $r$ $\geqslant 0$):
    \begin{small}
    \begin{displaymath}
      (\mathcal S)\, \left\lbrace\begin{array}{rcl}
        9\,a_1+ 10\,a_2+ 14\,a_3+ 15\,a_4+ 17\,a_5+ 21\,a_6+ 22 \,a_7 \hfill&=& u + r\,,\\[0.2cm]
	2\,(b_1+ b_2+ b_3+ b_4+ b_5+ b_6)+ 4\,(b_7+ b_8+ b_9+ b_{10}+ b_{11})+\hfill&&\\
        6\,(b_{12}+ b_{13}+ b_{14}+ b_{15}+ b_{16})+8\,(b_{17}+ b_{18}+
        b_{19})+10\,b_{20}+ 12\,b_{21} &=& v+r\,.
      \end{array}\right.
    \end{displaymath}
  \end{small}

  There are numerous works on how to compute minimal solutions of linear
  integer systems~\cite{CF1990,CD1991}.  Further, there exist optimized and
  reliable implementations, typically \textsc{normaliz} used in
  \textsc{macaulay~2}~\cite{BIS2015},
  or \textsc{4ti2} and especially, the so-called program
  \textsc{hilbert}~\cite{4ti2}.  In our case, we have only 2 linear equations,
  \textsc{hilbert} performs better than \textsc{normaliz}.  But the system for
  $\cov{\Sn{9}}$ has so many minimal solutions that we had to abort
  calculations after several hours of computation.

  Following Section~\ref{sec:inject-comp-line}, we regroup variables with
  same coefficients in $(\mathcal S)$, \textit{i.e.}  $\beta_1 = b_1+ b_2+
  b_3+ b_4+ b_5+ b_6$, $\beta_2 = b_7+ b_8+ b_9+ b_{10}+ b_{11}$, \textit{etc.} and
  we consider the injective companion of $(\mathcal{S})$,
  \begin{displaymath}
    (\tilde{\mathcal S})\,\left\lbrace\begin{array}{lcl}
        9\,a_1+ 10\,a_2+ 14\,a_3+ 15\,a_4+ 17\,a_5+ 21\,a_6+ 22 \,a_7 &=& u + r\,,\\
        2\,\beta_1 + 4\,\beta_2+ 6\,\beta_3+ 8\,\beta_4+10\,\beta_{5}+ 12\,\beta_{6} &=& v+r\,.
      \end{array}\right.
  \end{displaymath}
  This time, the software \textsc{hilbert} returned $7338$ solutions in only
  $25$ seconds on a laptop. From Remark~\ref{rem:inject-comp-line-1}, we
  finally found that the 7338 solutions of $(\tilde{\mathcal S})$ yield
  {58\,525\,823 minimal solutions} of $({\mathcal S})$.

  \begin{table}[htbp]
    \centering
    \begin{scriptsize}
      \begin{displaymath}\setlength{\arraycolsep}{2pt} \begin{array}{c||ccccccccccccccccc}
          d/o &   0 &   1 &  2  &   3 &   4 &  5  &  6  &   7 &   8 &  9  &   10 &   11 &  12  &  13  &  14 & \ldots \\\hline\hline
          1 &   - &   - &   - &   - &   - &   - &   - &   - &   - &   \checkmark &    - &    - &    - &    - &   - &  \ldots \\
          2 &   - &   - &   - &   - &   - &   - &   \checkmark &   - &   - &   - &    \checkmark &    - &    \checkmark &    - &   - &  \ldots \\
          3 &   - &   - &   - &   \checkmark &   - &   \checkmark &   - &   \checkmark &   - &   \checkmark &    - &    \checkmark &    - &    \checkmark &   - &  \ldots \\
          4 &   - &   - &   - &   - &   \checkmark &   - &   \checkmark &   - &   \checkmark &   - &    \checkmark &    - &    \checkmark &    - &   \checkmark &  \ldots \\
          5 &   - &   \checkmark &   - &   \checkmark &   - &   \checkmark &   - &   \checkmark &   - &   \checkmark &    - &    \checkmark &    - &    \checkmark &   - &  \ldots \\
          \vdots &   \vdots &   \vdots &   \vdots &   \vdots &   \vdots &   \vdots &   \vdots &   \vdots &   \vdots &   \vdots &    \vdots &    \vdots &    \vdots &    \vdots &   \vdots &  \ldots \\
          268 &   \checkmark &   - &   \checkmark &   - &   - &   - &   - &   - &   - &   - &    - &    - &    - &    - &   - &  \ldots \\
          \vdots &   \vdots &   \vdots &   \vdots &   \vdots &   \vdots &   \vdots &   \vdots &   \vdots &   \vdots &   \vdots &    \vdots &    \vdots &    \vdots &    \vdots &   \vdots &  \ldots \\
          502 &   \checkmark &   - &   - &   - &   - &   - &   - &   - &   - &   - &    - &    - &    - &    - &   - &  \ldots \\
          506 &   \checkmark &   - &   - &   - &   - &   - &   - &   - &   - &   - &    - &    - &    - &    - &   - &  \ldots \\
          510 &   \checkmark &   - &   - &   - &   - &   - &   - &   - &   - &   - &    - &    - &    - &    - &   - &  \ldots \\
        \end{array}
      \end{displaymath}
    \end{scriptsize}
    \medskip
    \caption{Couples $(d,m)$ for the family $\mathrm{A}_3$}
    \label{tab:taba3}
  \end{table}

  These $58\,525\,823$ minimal solutions yield transvectants that we can
  gather by degree $d$ and order $m$. It covers {1836} couples $(d,m)$
  (\textit{cf.}  Table~\ref{tab:taba3}).  For instance, the last row corresponds
  to the transvectant
  \begin{math}
    \tr{{\cc_{15}}^2}{{\mathsf{C}_{24, 2}}^{21}}{42}
  \end{math}
  where $\cc_{15}$ is the covariant of degree 3 and order 21 defined in
  Appendix~\ref{sec:mathbfcov_9} and $\mathsf{C}_{24, 2}$ is the covariant of
  degree 12 and order 2 in the table of $\Sn{6}$ evaluated at
  $\cc_3$ (\textit{cf.} Section~\ref{sec:gordans-iterations})\, which lead to a covariant of degree $24$ in $\ff$.
  \bigskip

  Following Section~\ref{sec:reform-theor-refthm}, we have now to compute the
  dimensions of these 1836 homogeneous spaces $\mathbf{Cov}_{d,m}(\Sn{9})$.
  But Theorem~\ref{thm:springer} yields for instance
  \begin{math}
    \dim \mathbf{Cov}_{501,0}(\Sn{9})= 14\,\,510\,116\,319\,,
  \end{math}
  which is far too large to be checked.

\subsubsection{Degree and order upper-bounds}
\label{sec:degree-order-bounds}

  Now, the upper-bounds of Section~\ref{sec:bounds-degr-orders} help to simplify a
  lot the computations.
  \begin{itemize}
  \item Using Lemma~\ref{lem:Order_Bounds_GY}, we can restrict to orders $m
    \leqslant 22$.
  \item Using Lemma~\ref{lem:degree_Bound_S9}, we have degree upper-bounds for
    every order $m\leqslant 22$.
  \end{itemize}

  Finally, $\mathbf{Cov}_{64,18}(\Sn{9})$ is one of the largest case
  that remains. Now,
  \begin{math}
    \dim \mathbf{Cov}_{64,18}(\Sn{9})= 1\,576\,149\,
  \end{math}
  is much smaller than $\dim \mathbf{Cov}_{501,0}(\Sn{9})$, but it is obviously
  still too large for Algorithm~\ref{fig:linearalgebra}.\medskip

  \subsubsection{Relations}
  \label{sec:relations}

  Following Section~\ref{sec:reductions}, we looked for relations between
  covariants of $\mathrm{B}_2$, which is the covariant basis of $\Sn{6}$.  Let us order the covariants of $\mathrm{B}_2$ first by degree then by order, invariants last,
  \begin{displaymath}
    \mathsf{C}_{24,2} > \mathsf{C}_{20, 2} > \mathsf{C}_{18, 4} > \mathsf{C}_{16, 2} > \ldots > \mathsf{C}_{4, 8} > \mathsf{C}_{2,6} >
    \mathsf{C}_{30,0} > \mathsf{C}_{10,0} > \mathsf{C}_{12,0} > \mathsf{C}_{8,0} > \mathsf{C}_{4,0}
  \end{displaymath}
  (we write $\mathsf{C}_{2d',m}$ for the covariant of degree $d'$ in $\cc_{3}\in \Sn{6}$ and order
  $m$, all being taken from the classical covariant basis of $\Sn{6}$ given for example
  in~\cite{GY2010}).\smallskip

  We found 18 relations of the form
  \begin{displaymath}
    \mathsf{C}_{d,m}^e = \sum \prod_{\mathsf{C} < \mathsf{C}_{d,m}} \mathsf{C}\,
  \end{displaymath}
  (where the power $e$ goes from $2$ for $\mathsf{C}_{24,2}$ up to $e=9$ for $\mathsf{C}_{6,8}$)
  and for any $\mathsf{C}_1> \mathsf{C}_2$ several hundred relations of the form
 \begin{displaymath}
   \mathsf{C}_1^{e_1}\times \mathsf{C}_2^{e_2}  = \sum \prod_{\mathsf{C} < \mathsf{C}_2} \mathsf{C}\,.
 \end{displaymath}

 Thanks to the degree/order upper-bounds and the relations for $\mathrm{B}_2$, only
 {$235\,493$} transvectants remain.  It decreases the number of spaces
 $\mathbf{Cov}_{d,m}(\Sn{9})$ to be tested to 633 (instead of 1836).  The
 largest one is $\mathbf{Cov}_{60,14}(\Sn{9})$, its dimension is about 2 times
 smaller than $\mathbf{Cov}_{66,18}(\Sn{9})$,
 \begin{displaymath}
   {\dim \mathbf{Cov}_{60,14}(\Sn{9}) = 872\,368},
 \end{displaymath}
 but it is still slightly too large for Algorithm~\ref{fig:linearalgebra} (the
 complexity of which is at least cubic in this dimension).

 \subsubsection{Reductions}
 \label{sec:reduction}

 Now, reductions by primary invariants enable to conclude (\textit{cf.}
 Section~\ref{sec:reductions}).  Let ${\mathfrak p}_4$, ${\mathfrak q}_4$ and
 ${\mathfrak p}_8$ (resp. denoted $\cc_{16}$, $\cc_{17}$ and $\cc_{121}$ in
 Appendix~\ref{sec:mathbfcov_9}) be the invariants defined by
 Proposition~\ref{prop:hsop_Inv9}: they are the first three generators of a
 h.s.o.p of degrees $4$, $4$, $8$, $12$, $14$, $16$, $30$ for
 $\inv{\Sn{9}}$. So, instead of $\mathbf{Cov}_{d,m}({\Sn{9}})$, consider the
 quotient
 \begin{displaymath}
   {\mathcal Q}_{d,m} := \mathbf{Cov}_{d,m}(\Sn{9})\,/\,({{\mathfrak p}_4}\,\mathbf{Cov}_{d-4,m}({\Sn{9})} + {{\mathfrak q}_4}\,\mathbf{Cov}_{d-4,m}{(\Sn{9})}+ {{\mathfrak p}_8}\,\mathbf{Cov}_{d-8,m}{(\Sn{9})})\,.
 \end{displaymath}
 Note that working in ${\mathcal Q}_{d,m}$ amounts to evaluate covariants at
 random forms $\ff$ that zeroify ${\mathfrak p}_4$, ${\mathfrak q}_4$ and ${\mathfrak p}_8$ at Step (2) of
 Algorithm~\ref{fig:linearalgebra}.  
 Furthermore, we can derive from Equation~(\ref{eq:Hilbert_S_Quotient}) the relation
 \begin{multline}\label{eq:2}
   \dim {\mathcal Q}_{d,m} = \dim \mathbf{Cov}_{d,m}({\Sn{9})}\ -\ 2\, \dim {
     \mathbf{Cov}_{d-4,m}({\Sn{9}})}\ +\\ 2\, \dim {\mathbf{Cov}_{d-12,m}({\Sn{9}})}\ -\ \dim {
     \mathbf{Cov}_{d-16,m}({\Sn{9}})}\,.
 \end{multline}
 Typically, we find
 \begin{math}
   {\dim \mathcal Q_{60,14} = 33\,360}\,,
 \end{math}
 which is finally affordable with Algorithm~\ref{fig:linearalgebra}.

 \subsubsection{Linear algebra}
 \label{sec:linear-algebra-1}

 Finally, we are left with 633 spaces ${\mathcal Q}_{d,m}$, the dimension of
 which must be checked versus what is predicted by Equation~(\ref{eq:2}).  These
 dimensions go from 1 for ${\mathcal Q}_{1,9}$ to 33\,360 for ${\mathcal
   Q}_{60,14}$\,.  We did it modulo $p=65521$. The whole computation took less
 than one day on a \textsc{dell} computer with 32 processors (\textsc{1400MHz
   AMD Opteron}).
 For instance, for ${\mathcal Q}_{60,14}$, it took three hours on one
 processor: two hours to compute the matrix ${\mathtt M}^{\top}$ in Step (3) of
 Algorithm~\ref{fig:linearalgebra} (its rank was 33\,359) and one extra hour
 to find the missing generator in Step (4).

 \subsection{Covariant basis of binary decimics}
 \label{sec:covar-basis-binary-1}

 As for $\cov{\Sn{9}}$, a (possibly incomplete) minimal basis for
 $\cov{\Sn{10}}$ have been already computed in the past
 (see~\cite{Brouwer2015}).
 With our implementation of Olver's algorithm, we retrieve a basis with $510$
 generators.  A complete but somehow unappealing definition for these
 generators is in Appendix~\ref{sec:mathbfcov_10}. We gather in
 Table~\ref{tab:cov10} the number of generators for each degree and each order
 too.\medskip

 \begin{table}[H]
   \centering
   \begin{footnotesize}
     \begin{math}\setlength{\arraycolsep}{1pt} \begin{array}{c||cccccccccccccc|c|c|c}
         d/o & 0 & 2 & 4 & 6 & 8 & 10 & 12 & 14 & 16 & 18 & 20 & 22 & 24 & 26
         & \# & Cum \\\hline\hline
         1 &   - &   - &   - &   - &   - &   1 &   - &   - &   - &   - &   - &   - &   - &   - &   1 &   1 \\
         2 &   1 &   - &   1 &   - &   1 &   - &   1 &   - &   1 &   - &   - &   - &   - &   - &   5 &   6 \\
         3 &   - &   1 &   - &   2 &   1 &   1 &   2 &   1 &   1 &   1 &   1 &   - &   1 &   - &  12 &  18 \\
         4 &   1 &   - &   3 &   1 &   3 &   3 &   2 &   3 &   1 &   2 &   1 &   1 &   - &   1 &  22 &  40 \\
         5 &   - &   3 &   3 &   4 &   5 &   4 &   5 &   2 &   4 &   - &   2 &   - &   - &   - &  32 &  72 \\
         6 &   4 &   2 &   5 &   8 &   6 &   8 &   2 &   4 &   - &   1 &   - &   - &   - &   - &  40 & 112 \\
         7 &   - &   7 &  10 &   8 &  12 &   2 &   4 &   - &   1 &   - &   - &   - &   - &   - &  44 & 156 \\
         8 &   5 &   8 &  11 &  15 &   4 &   7 &   - &   1 &   - &   - &   - &   - &   - &   - &  51 & 207 \\
         9 &   5 &  13 &  19 &   8 &   7 &   - &   1 &   - &   - &   - &   - &   - &   - &   - &  53 & 260 \\
         10 &   8 &  20 &  13 &  13 &   - &   1 &   - &   - &   - &   - &   - &   - &   - &   - &  55 & 315 \\
         11 &   8 &  18 &  21 &   - &   1 &   - &   - &   - &   - &   - &   - &   - &   - &   - &  48 & 363 \\
         12 &  12 &  30 &   1 &   2 &   - &   - &   - &   - &   - &   - &   - &   - &   - &   - &  45 & 408 \\
         13 &  15 &  16 &   2 &   - &   - &   - &   - &   - &   - &   - &   - &   - &   - &   - &  33 & 441 \\
         14 &  13 &  17 &   - &   - &   - &   - &   - &   - &   - &   - &   - &   - &   - &   - &  30 & 471 \\
         15 &  19 &   - &   1 &   - &   - &   - &   - &   - &   - &   - &   - &   - &   - &   - &  20 & 491 \\
         16 &   5 &   3 &   - &   - &   - &   - &   - &   - &   - &   - &   - &   - &   - &   - &   8 & 499 \\
         17 &   5 &   - &   - &   - &   - &   - &   - &   - &   - &   - &   - &   - &   - &   - &   5 & 504 \\
         18 &   1 &   1 &   - &   - &   - &   - &   - &   - &   - &   - &   - &   - &   - &   - &   2 & 506 \\
         19 &   2 &   - &   - &   - &   - &   - &   - &   - &   - &   - &   - &   - &   - &   - &   2 & 508 \\
         20 &   - &   - &   - &   - &   - &   - &   - &   - &   - &   - &   - &   - &   - &   - &   - & 508 \\
         21 & 2 & - & - & - & - & - & - & - & - & - & - & - & - & - & 2 & 510
         \\\hline
         Tot & 106 & 139 &  90 &  61 &  40 &  27 &  17 &  11 &   8 &   4 &   4 &   1 &   1 &   1 & 510 &     \\
       \end{array}
     \end{math}
   \end{footnotesize}
   \medskip
   \caption{Minimal basis of $\mathbf{Cov}(\Sn{10})$}
   \label{tab:cov10}
 \end{table}

 The calculations that we have made to prove that this table is indeed
 complete are finally very similar to the ones for $\mathbf{Cov}(\Sn{9})$.
 The main difficulty is again the computation of $\mathrm{A}_3$, especially we have to
 deal with $69-9=60$ covariants for $\mathbf{Cov}(\Sn{8})$, of order $2$, $4$, $6$,
 $8$, $10$, $12$, $14$, $18$ (instead of the $21$ covariants of
 $\cov{\Sn{6}}$).
 The integer system $(\mathcal S)$ is
 \begin{small}
   \begin{displaymath}
     (\mathcal S)\ \left\lbrace\begin{array}{rcl}
     10\,a_1+ 12\,a_2+ 16\,a_3+ 18\,a_4+ 20\,a_5+ 24\,a_6+ 26 \,a_7 \hfill&=& u +
     r\,,\\[0.2cm]       
     2\,(b_1+\ldots+ b_{14})\,+\,
     4\,(b_{15}+ \ldots+ b_{27})\,+\, 6\,(b_{28}+ \ldots + b_{39})\,+\,
     8\,(b_{40}+\ldots+ b_{45})\,+&&\\
     \, 10\,(b_{46}+ \ldots + b_{52})\,+\,
     12\,(b_{53}+b_{54}+ b_{55})\,+\, 14\,(b_{56}+b_{57}+ b_{58})\,+\,
     18\,(b_{59}+b_{60}) &=& v+r\,.
     \end{array}\right.
   \end{displaymath}
 \end{small}
 It took here slightly less than $3$ minutes on a laptop to find the $8985$
 minimal solutions of the injective companion $(\tilde{\mathcal S})$ of
 $(\mathcal S)$, which in return yields $1\,345\,290\,951$ minimal solutions
 for $(\mathcal S)$\,.\medskip
 
 Relations and degree/order upper-bounds that we have for $\cov{\Sn{10}}$
 improve a lot the situation. Especially, the order can not be larger than
 $\lambda_{10}=26$ (see Lemma~\ref{lem:Order_Bounds_GY}) and the degree
 upper-bounds for medium size orders are slightly better than the ones of
 $\cov{\Sn{9}}$ (\textit{cf.}  Lemma~\ref{lem:degree_Bound_S10}).

    So, we finally arrive at {$588$ spaces} $\mathbf{Cov}_{d,m}(\Sn{10})$ to be
    checked.
    The largest one is $\mathbf{Cov}_{46,20}(\Sn{10})$, which is only of
    dimension 26323 if we work modulo the invariants $\mathfrak{p}'_{2},\mathfrak{p}'_{4},\mathfrak{p}'_{6},\mathfrak{q}'_{6}$ (resp. denoted $\cc_{2}$,
    $\cc_{19}$, $\cc_{73}$ and $\cc_{74}$ in Appendix~\ref{sec:mathbfcov_10})
    of degree 2, 4, 6 and 6 defined in the h.s.o.p. of $\inv{\Sn{10}}$~\cite{BP2010}
    (\textit{cf.}  Proposition~\ref{prop:hsop_Inv10}).\medskip

    The whole computation was finally slightly easier than $\cov{\Sn{9}}$,
    it took about 4 hours on the same \textsc{dell} computer.
    All in all, we have proved this theorem.
    \begin{thm}\label{thm:covar-basis-binary-10}
      The 510 covariants given in Appendix~\ref{sec:mathbfcov_10} define a minimal
      basis for the covariant algebra $\cov{\Sn{10}}$\,.
    \end{thm}
 

\section*{Acknowledgement}
We would like to thank Hanspeter Kraft for pointing out the work of Michel Van
den Bergh to us.  We wish to express our grateful
acknowledgment to Boris Kolev and Christophe Ritzentahler for fruitful
discussions on preliminary versions of the paper too.


\pagebreak[1]
\appendix

\section{A minimal basis for $\cov{\Sn{9}}$}
\label{sec:mathbfcov_9}

\setlength{\columnseprule}{1pt}

\begin{tiny}
  \begin{multicols}{6}
\noindent\textbf{Degree} 1:\\
\noindent\hspace*{0.2cm}\textbf{Order} 9:\\
${\cc_{1}=\ff}$\\

\noindent\textbf{Degree} 2:\\
\noindent\hspace*{0.2cm}\textbf{Order} 2:\\
${\cc_{2}=(\ff,\ff)_8}$\\
\noindent\hspace*{0.2cm}\textbf{Order} 6:\\
${\cc_{3}=(\ff,\ff)_6}$\\
\noindent\hspace*{0.2cm}\textbf{Order} 10:\\
${\cc_{4}=(\ff,\ff)_4}$\\
\noindent\hspace*{0.2cm}\textbf{Order} 14:\\
${\cc_{5}=(\ff,\ff)_2}$\\

\noindent\textbf{Degree} 3:\\
\noindent\hspace*{0.2cm}\textbf{Order} 3:\\
${\cc_{6}=(\cc_{4},\ff)_8}$\\
\noindent\hspace*{0.2cm}\textbf{Order} 5:\\
${\cc_{7}=(\cc_{5},\ff)_9}$\\
\noindent\hspace*{0.2cm}\textbf{Order} 7:\\
${\cc_{8}=(\cc_{5},\ff)_8}$\\
\noindent\hspace*{0.2cm}\textbf{Order} 9:\\
${\cc_{9}=(\cc_{5},\ff)_7}$\\
${\cc_{10}=(\cc_{4},\ff)_5}$\\
\noindent\hspace*{0.2cm}\textbf{Order} 11:\\
${\cc_{11}=(\cc_{5},\ff)_6}$\\
\noindent\hspace*{0.2cm}\textbf{Order} 13:\\
${\cc_{12}=(\cc_{5},\ff)_5}$\\
\noindent\hspace*{0.2cm}\textbf{Order} 15:\\
${\cc_{13}=(\cc_{5},\ff)_4}$\\
\noindent\hspace*{0.2cm}\textbf{Order} 17:\\
${\cc_{14}=(\cc_{5},\ff)_3}$\\
\noindent\hspace*{0.2cm}\textbf{Order} 21:\\
${\cc_{15}=(\cc_{5},\ff)_1}$\\

\noindent\textbf{Degree} 4:\\
\noindent\hspace*{0.2cm}\textbf{Order} 0:\\
${\cc_{16}=(\cc_{2},\cc_{2})_2}$\\
${\cc_{17}=(\cc_{3},\cc_{3})_6}$\\
\noindent\hspace*{0.2cm}\textbf{Order} 4:\\
${\cc_{18}=(\cc_{12},\ff)_9}$\\
${\cc_{19}=(\cc_{11},\ff)_8}$\\
\noindent\hspace*{0.2cm}\textbf{Order} 6:\\
${\cc_{20}=(\cc_{13},\ff)_9}$\\
${\cc_{21}=(\cc_{12},\ff)_8}$\\
\noindent\hspace*{0.2cm}\textbf{Order} 8:\\
${\cc_{22}=(\cc_{14},\ff)_9}$\\
${\cc_{23}=(\cc_{13},\ff)_8}$\\
${\cc_{24}=(\cc_{12},\ff)_7}$\\
\noindent\hspace*{0.2cm}\textbf{Order} 10:\\
${\cc_{25}=(\cc_{14},\ff)_8}$\\
${\cc_{26}=(\cc_{13},\ff)_7}$\\
\noindent\hspace*{0.2cm}\textbf{Order} 12:\\
${\cc_{27}=(\cc_{15},\ff)_9}$\\
${\cc_{28}=(\cc_{14},\ff)_7}$\\
\noindent\hspace*{0.2cm}\textbf{Order} 14:\\
${\cc_{29}=(\cc_{15},\ff)_8}$\\
${\cc_{30}=(\cc_{14},\ff)_6}$\\
\noindent\hspace*{0.2cm}\textbf{Order} 16:\\
${\cc_{31}=(\cc_{15},\ff)_7}$\\
\noindent\hspace*{0.2cm}\textbf{Order} 18:\\
${\cc_{32}=(\cc_{15},\ff)_6}$\\
\noindent\hspace*{0.2cm}\textbf{Order} 22:\\
${\cc_{33}=(\cc_{15},\ff)_4}$\\

\noindent\textbf{Degree} 5:\\
\noindent\hspace*{0.2cm}\textbf{Order} 1:\\
${\cc_{34}=(\cc_{26},\ff)_9}$\\
\noindent\hspace*{0.2cm}\textbf{Order} 3:\\
${\cc_{35}=(\cc_{28},\ff)_9}$\\
${\cc_{36}=(\cc_{27},\ff)_9}$\\
${\cc_{37}=(\cc_{26},\ff)_8}$\\
\noindent\hspace*{0.2cm}\textbf{Order} 5:\\
${\cc_{38}=(\cc_{30},\ff)_9}$\\
${\cc_{39}=(\cc_{29},\ff)_9}$\\
${\cc_{40}=(\cc_{28},\ff)_8}$\\
${\cc_{41}=(\cc_{26},\ff)_7}$\\
\noindent\hspace*{0.2cm}\textbf{Order} 7:\\
${\cc_{42}=(\cc_{31},\ff)_9}$\\
${\cc_{43}=(\cc_{30},\ff)_8}$\\
${\cc_{44}=(\cc_{29},\ff)_8}$\\
${\cc_{45}=(\cc_{28},\ff)_7}$\\
\noindent\hspace*{0.2cm}\textbf{Order} 9:\\
${\cc_{46}=(\cc_{32},\ff)_9}$\\
${\cc_{47}=(\cc_{31},\ff)_8}$\\
${\cc_{48}=(\cc_{30},\ff)_7}$\\
\noindent\hspace*{0.2cm}\textbf{Order} 11:\\
${\cc_{49}=(\cc_{32},\ff)_8}$\\
${\cc_{50}=(\cc_{31},\ff)_7}$\\
${\cc_{51}=(\cc_{30},\ff)_6}$\\
${\cc_{52}=(\cc_{29},\ff)_6}$\\
\noindent\hspace*{0.2cm}\textbf{Order} 13:\\
${\cc_{53}=(\cc_{33},\ff)_9}$\\
${\cc_{54}=(\cc_{32},\ff)_7}$\\
\noindent\hspace*{0.2cm}\textbf{Order} 15:\\
${\cc_{55}=(\cc_{33},\ff)_8}$\\
${\cc_{56}=(\cc_{32},\ff)_6}$\\
${\cc_{57}=(\cc_{31},\ff)_5}$\\
\noindent\hspace*{0.2cm}\textbf{Order} 19:\\
${\cc_{58}=(\cc_{33},\ff)_6}$\\

\noindent\textbf{Degree} 6:\\
\noindent\hspace*{0.2cm}\textbf{Order} 2:\\
${\cc_{59}=(\cc_{52},\ff)_9}$\\
${\cc_{60}=(\cc_{51},\ff)_9}$\\
${\cc_{61}=(\cc_{50},\ff)_9}$\\
${\cc_{62}=(\cc_{49},\ff)_9}$\\
\noindent\hspace*{0.2cm}\textbf{Order} 4:\\
${\cc_{63}=(\cc_{54},\ff)_9}$\\
${\cc_{64}=(\cc_{53},\ff)_9}$\\
${\cc_{65}=(\cc_{52},\ff)_8}$\\
${\cc_{66}=(\cc_{51},\ff)_8}$\\
\noindent\hspace*{0.2cm}\textbf{Order} 6:\\
${\cc_{67}=(\cc_{57},\ff)_9}$\\
${\cc_{68}=(\cc_{56},\ff)_9}$\\
${\cc_{69}=(\cc_{55},\ff)_9}$\\
${\cc_{70}=(\cc_{54},\ff)_8}$\\
${\cc_{71}=(\cc_{52},\ff)_7}$\\
${\cc_{72}=(\cc_{51},\ff)_7}$\\
\noindent\hspace*{0.2cm}\textbf{Order} 8:\\
${\cc_{73}=(\cc_{57},\ff)_8}$\\
${\cc_{74}=(\cc_{56},\ff)_8}$\\
${\cc_{75}=(\cc_{55},\ff)_8}$\\
${\cc_{76}=(\cc_{54},\ff)_7}$\\
${\cc_{77}=(\cc_{53},\ff)_7}$\\
${\cc_{78}=(\cc_{52},\ff)_6}$\\
\noindent\hspace*{0.2cm}\textbf{Order} 10:\\
${\cc_{79}=(\cc_{58},\ff)_9}$\\
${\cc_{80}=(\cc_{57},\ff)_7}$\\
${\cc_{81}=(\cc_{56},\ff)_7}$\\
\noindent\hspace*{0.2cm}\textbf{Order} 12:\\
${\cc_{82}=(\cc_{58},\ff)_8}$\\
${\cc_{83}=(\cc_{57},\ff)_6}$\\
${\cc_{84}=(\cc_{56},\ff)_6}$\\
${\cc_{85}=(\cc_{55},\ff)_6}$\\
\noindent\hspace*{0.2cm}\textbf{Order} 16:\\
${\cc_{86}=(\cc_{58},\ff)_6}$\\

\noindent\textbf{Degree} 7:\\
\noindent\hspace*{0.2cm}\textbf{Order} 1:\\
${\cc_{87}=(\cc_{81},\ff)_9}$\\
${\cc_{88}=(\cc_{80},\ff)_9}$\\
${\cc_{89}=(\cc_{79},\ff)_9}$\\
${\cc_{90}=(\cc_{78},\ff)_8}$\\
\noindent\hspace*{0.2cm}\textbf{Order} 3:\\
${\cc_{91}=(\cc_{85},\ff)_9}$\\
${\cc_{92}=(\cc_{84},\ff)_9}$\\
${\cc_{93}=(\cc_{83},\ff)_9}$\\
${\cc_{94}=(\cc_{82},\ff)_9}$\\
${\cc_{95}=(\cc_{81},\ff)_8}$\\
${\cc_{96}=(\cc_{80},\ff)_8}$\\
${\cc_{97}=(\cc_{78},\ff)_7}$\\
\noindent\hspace*{0.2cm}\textbf{Order} 5:\\
${\cc_{98}=(\cc_{85},\ff)_8}$\\
${\cc_{99}=(\cc_{84},\ff)_8}$\\
${\cc_{100}=(\cc_{83},\ff)_8}$\\
${\cc_{101}=(\cc_{82},\ff)_8}$\\
${\cc_{102}=(\cc_{81},\ff)_7}$\\
${\cc_{103}=(\cc_{80},\ff)_7}$\\
${\cc_{104}=(\cc_{79},\ff)_7}$\\
${\cc_{105}=(\cc_{78},\ff)_6}$\\
\noindent\hspace*{0.2cm}\textbf{Order} 7:\\
${\cc_{106}=(\cc_{86},\ff)_9}$\\
${\cc_{107}=(\cc_{85},\ff)_7}$\\
${\cc_{108}=(\cc_{84},\ff)_7}$\\
${\cc_{109}=(\cc_{83},\ff)_7}$\\
${\cc_{110}=(\cc_{82},\ff)_7}$\\
${\cc_{111}=(\cc_{81},\ff)_6}$\\
${\cc_{112}=(\cc_{80},\ff)_6}$\\
\noindent\hspace*{0.2cm}\textbf{Order} 9:\\
${\cc_{113}=(\cc_{86},\ff)_8}$\\
${\cc_{114}=(\cc_{85},\ff)_6}$\\
${\cc_{115}=(\cc_{84},\ff)_6}$\\
${\cc_{116}=(\cc_{83},\ff)_6}$\\
${\cc_{117}=(\cc_{82},\ff)_6}$\\
${\cc_{118}=(\cc_{81},\ff)_5}$\\
\noindent\hspace*{0.2cm}\textbf{Order} 11:\\
${\cc_{119}=(\cc_{86},\ff)_7}$\\
\noindent\hspace*{0.2cm}\textbf{Order} 13:\\
${\cc_{120}=(\cc_{86},\ff)_6}$\\

\noindent\textbf{Degree} 8:\\
\noindent\hspace*{0.2cm}\textbf{Order} 0:\\
${\cc_{121}=(\cc_{2}^3,\cc_{3})_6}$\\
${\cc_{122}=(\cc_{118},\ff)_9}$\\
${\cc_{123}=(\cc_{7}\cc_{19},\ff)_9}$\\
${\cc_{124}=(\cc_{7}\cc_{18},\ff)_9}$\\
${\cc_{125}=(\cc_{6}\cc_{21},\ff)_9}$\\
\noindent\hspace*{0.2cm}\textbf{Order} 2:\\
${\cc_{126}=(\cc_{119},\ff)_9}$\\
${\cc_{127}=(\cc_{118},\ff)_8}$\\
${\cc_{128}=(\cc_{117},\ff)_8}$\\
${\cc_{129}=(\cc_{116},\ff)_8}$\\
${\cc_{130}=(\cc_{115},\ff)_8}$\\
${\cc_{131}=(\cc_{114},\ff)_8}$\\
${\cc_{132}=(\cc_{113},\ff)_8}$\\
${\cc_{133}=(\cc_{112},\ff)_7}$\\
\noindent\hspace*{0.2cm}\textbf{Order} 4:\\
${\cc_{134}=(\cc_{120},\ff)_9}$\\
${\cc_{135}=(\cc_{119},\ff)_8}$\\
${\cc_{136}=(\cc_{118},\ff)_7}$\\
${\cc_{137}=(\cc_{117},\ff)_7}$\\
${\cc_{138}=(\cc_{116},\ff)_7}$\\
${\cc_{139}=(\cc_{115},\ff)_7}$\\
${\cc_{140}=(\cc_{114},\ff)_7}$\\
${\cc_{141}=(\cc_{112},\ff)_6}$\\
${\cc_{142}=(\cc_{111},\ff)_6}$\\
${\cc_{143}=(\cc_{110},\ff)_6}$\\
\noindent\hspace*{0.2cm}\textbf{Order} 6:\\
${\cc_{144}=(\cc_{120},\ff)_8}$\\
${\cc_{145}=(\cc_{119},\ff)_7}$\\
${\cc_{146}=(\cc_{118},\ff)_6}$\\
${\cc_{147}=(\cc_{117},\ff)_6}$\\
${\cc_{148}=(\cc_{116},\ff)_6}$\\
${\cc_{149}=(\cc_{115},\ff)_6}$\\
${\cc_{150}=(\cc_{114},\ff)_6}$\\
${\cc_{151}=(\cc_{113},\ff)_6}$\\
${\cc_{152}=(\cc_{112},\ff)_5}$\\
${\cc_{153}=(\cc_{111},\ff)_5}$\\
\noindent\hspace*{0.2cm}\textbf{Order} 8:\\
${\cc_{154}=(\cc_{120},\ff)_7}$\\
${\cc_{155}=(\cc_{119},\ff)_6}$\\
${\cc_{156}=(\cc_{118},\ff)_5}$\\
${\cc_{157}=(\cc_{117},\ff)_5}$\\
\noindent\hspace*{0.2cm}\textbf{Order} 10:\\
${\cc_{158}=(\cc_{120},\ff)_6}$\\
${\cc_{159}=(\cc_{119},\ff)_5}$\\

\noindent\textbf{Degree} 9:\\
\noindent\hspace*{0.2cm}\textbf{Order} 1:\\
${\cc_{160}=(\cc_{159},\ff)_9}$\\
${\cc_{161}=(\cc_{158},\ff)_9}$\\
${\cc_{162}=(\cc_{157},\ff)_8}$\\
${\cc_{163}=(\cc_{156},\ff)_8}$\\
${\cc_{164}=(\cc_{19}\cc_{21},\ff)_9}$\\
${\cc_{165}=(\cc_{19}\cc_{20},\ff)_9}$\\
${\cc_{166}=(\cc_{19}^2,\ff)_8}$\\
${\cc_{167}=(\cc_{18}\cc_{21},\ff)_9}$\\
${\cc_{168}=(\cc_{18}\cc_{20},\ff)_9}$\\
\noindent\hspace*{0.2cm}\textbf{Order} 3:\\
${\cc_{169}=(\cc_{159},\ff)_8}$\\
${\cc_{170}=(\cc_{158},\ff)_8}$\\
${\cc_{171}=(\cc_{157},\ff)_7}$\\
${\cc_{172}=(\cc_{156},\ff)_7}$\\
${\cc_{173}=(\cc_{155},\ff)_7}$\\
${\cc_{174}=(\cc_{154},\ff)_7}$\\
${\cc_{175}=(\cc_{153},\ff)_6}$\\
${\cc_{176}=(\cc_{152},\ff)_6}$\\
${\cc_{177}=(\cc_{151},\ff)_6}$\\
${\cc_{178}=(\cc_{150},\ff)_6}$\\
${\cc_{179}=(\cc_{149},\ff)_6}$\\
${\cc_{180}=(\cc_{148},\ff)_6}$\\
${\cc_{181}=(\cc_{147},\ff)_6}$\\
${\cc_{182}=(\cc_{146},\ff)_6}$\\
\noindent\hspace*{0.2cm}\textbf{Order} 5:\\
${\cc_{183}=(\cc_{159},\ff)_7}$\\
${\cc_{184}=(\cc_{158},\ff)_7}$\\
${\cc_{185}=(\cc_{157},\ff)_6}$\\
${\cc_{186}=(\cc_{156},\ff)_6}$\\
${\cc_{187}=(\cc_{155},\ff)_6}$\\
${\cc_{188}=(\cc_{154},\ff)_6}$\\
${\cc_{189}=(\cc_{153},\ff)_5}$\\
${\cc_{190}=(\cc_{152},\ff)_5}$\\
${\cc_{191}=(\cc_{151},\ff)_5}$\\
${\cc_{192}=(\cc_{150},\ff)_5}$\\
\noindent\hspace*{0.2cm}\textbf{Order} 7:\\
${\cc_{193}=(\cc_{159},\ff)_6}$\\
${\cc_{194}=(\cc_{158},\ff)_6}$\\
${\cc_{195}=(\cc_{157},\ff)_5}$\\
${\cc_{196}=(\cc_{156},\ff)_5}$\\
${\cc_{197}=(\cc_{155},\ff)_5}$\\
${\cc_{198}=(\cc_{154},\ff)_5}$\\
${\cc_{199}=(\cc_{153},\ff)_4}$\\
\noindent\hspace*{0.2cm}\textbf{Order} 9:\\
${\cc_{200}=(\cc_{159},\ff)_5}$\\

\noindent\textbf{Degree} 10:\\
\noindent\hspace*{0.2cm}\textbf{Order} 0:\\
${\cc_{201}=(\cc_{200},\ff)_9}$\\
${\cc_{202}=(\cc_{37}\cc_{21},\ff)_9}$\\
${\cc_{203}=(\cc_{36}\cc_{21},\ff)_9}$\\
${\cc_{204}=(\cc_{37}\cc_{20},\ff)_9}$\\
${\cc_{205}=(\cc_{19}\cc_{41},\ff)_9}$\\
\noindent\hspace*{0.2cm}\textbf{Order} 2:\\
${\cc_{206}=(\cc_{200},\ff)_8}$\\
${\cc_{207}=(\cc_{199},\ff)_7}$\\
${\cc_{208}=(\cc_{198},\ff)_7}$\\
${\cc_{209}=(\cc_{197},\ff)_7}$\\
${\cc_{210}=(\cc_{196},\ff)_7}$\\
${\cc_{211}=(\cc_{195},\ff)_7}$\\
${\cc_{212}=(\cc_{194},\ff)_7}$\\
${\cc_{213}=(\cc_{24}\cc_{37},\ff)_9}$\\
${\cc_{214}=(\cc_{24}\cc_{36},\ff)_9}$\\
${\cc_{215}=(\cc_{24}\cc_{35},\ff)_9}$\\
${\cc_{216}=(\cc_{34}\cc_{24},\ff)_8}$\\
${\cc_{217}=(\cc_{23}\cc_{37},\ff)_9}$\\
${\cc_{218}=(\cc_{23}\cc_{36},\ff)_9}$\\
${\cc_{219}=(\cc_{23}\cc_{35},\ff)_9}$\\
${\cc_{220}=(\cc_{23}\cc_{34},\ff)_8}$\\
\noindent\hspace*{0.2cm}\textbf{Order} 4:\\
${\cc_{221}=(\cc_{200},\ff)_7}$\\
${\cc_{222}=(\cc_{199},\ff)_6}$\\
${\cc_{223}=(\cc_{198},\ff)_6}$\\
${\cc_{224}=(\cc_{197},\ff)_6}$\\
${\cc_{225}=(\cc_{196},\ff)_6}$\\
${\cc_{226}=(\cc_{195},\ff)_6}$\\
${\cc_{227}=(\cc_{194},\ff)_6}$\\
${\cc_{228}=(\cc_{193},\ff)_6}$\\
${\cc_{229}=(\cc_{192},\ff)_5}$\\
${\cc_{230}=(\cc_{191},\ff)_5}$\\
${\cc_{231}=(\cc_{190},\ff)_5}$\\
${\cc_{232}=(\cc_{189},\ff)_5}$\\
${\cc_{233}=(\cc_{188},\ff)_5}$\\
${\cc_{234}=(\cc_{187},\ff)_5}$\\
${\cc_{235}=(\cc_{186},\ff)_5}$\\
\noindent\hspace*{0.2cm}\textbf{Order} 6:\\
${\cc_{236}=(\cc_{200},\ff)_6}$\\
${\cc_{237}=(\cc_{199},\ff)_5}$\\
${\cc_{238}=(\cc_{198},\ff)_5}$\\
\noindent\hspace*{0.2cm}\textbf{Order} 8:\\
${\cc_{239}=(\cc_{200},\ff)_5}$\\

\noindent\textbf{Degree} 11:\\
\noindent\hspace*{0.2cm}\textbf{Order} 1:\\
${\cc_{240}=(\cc_{239},\ff)_8}$\\
${\cc_{241}=(\cc_{41}^2,\ff)_9}$\\
${\cc_{242}=(\cc_{40}\cc_{41},\ff)_9}$\\
${\cc_{243}=(\cc_{40}^2,\ff)_9}$\\
${\cc_{244}=(\cc_{39}\cc_{41},\ff)_9}$\\
${\cc_{245}=(\cc_{39}\cc_{40},\ff)_9}$\\
${\cc_{246}=(\cc_{39}^2,\ff)_9}$\\
${\cc_{247}=(\cc_{38}\cc_{41},\ff)_9}$\\
${\cc_{248}=(\cc_{38}\cc_{40},\ff)_9}$\\
${\cc_{249}=(\cc_{38}\cc_{39},\ff)_9}$\\
${\cc_{250}=(\cc_{38}^2,\ff)_9}$\\
${\cc_{251}=(\cc_{45}\cc_{37},\ff)_9}$\\
${\cc_{252}=(\cc_{44}\cc_{37},\ff)_9}$\\
${\cc_{253}=(\cc_{37}\cc_{43},\ff)_9}$\\
${\cc_{254}=(\cc_{37}\cc_{42},\ff)_9}$\\
${\cc_{255}=(\cc_{37}\cc_{41},\ff)_8}$\\
${\cc_{256}=(\cc_{37}\cc_{40},\ff)_8}$\\
\noindent\hspace*{0.2cm}\textbf{Order} 3:\\
${\cc_{257}=(\cc_{239},\ff)_7}$\\
${\cc_{258}=(\cc_{238},\ff)_6}$\\
${\cc_{259}=(\cc_{237},\ff)_6}$\\
${\cc_{260}=(\cc_{236},\ff)_6}$\\
${\cc_{261}=(\cc_{45}\cc_{41},\ff)_9}$\\
${\cc_{262}=(\cc_{44}\cc_{41},\ff)_9}$\\
${\cc_{263}=(\cc_{41}\cc_{43},\ff)_9}$\\
${\cc_{264}=(\cc_{41}\cc_{42},\ff)_9}$\\
${\cc_{265}=(\cc_{41}^2,\ff)_8}$\\
${\cc_{266}=(\cc_{45}\cc_{40},\ff)_9}$\\
${\cc_{267}=(\cc_{44}\cc_{40},\ff)_9}$\\
${\cc_{268}=(\cc_{40}\cc_{43},\ff)_9}$\\
${\cc_{269}=(\cc_{40}\cc_{42},\ff)_9}$\\
${\cc_{270}=(\cc_{40}\cc_{41},\ff)_8}$\\
${\cc_{271}=(\cc_{40}^2,\ff)_8}$\\
${\cc_{272}=(\cc_{45}\cc_{39},\ff)_9}$\\
\noindent\hspace*{0.2cm}\textbf{Order} 5:\\
${\cc_{273}=(\cc_{239},\ff)_6}$\\
${\cc_{274}=(\cc_{238},\ff)_5}$\\
${\cc_{275}=(\cc_{237},\ff)_5}$\\
${\cc_{276}=(\cc_{236},\ff)_5}$\\
${\cc_{277}=(\cc_{235},\ff)_4}$\\
${\cc_{278}=(\cc_{234},\ff)_4}$\\
${\cc_{279}=(\cc_{233},\ff)_4}$\\
\noindent\hspace*{0.2cm}\textbf{Order} 7:\\
${\cc_{280}=(\cc_{239},\ff)_5}$\\

\noindent\textbf{Degree} 12:\\
\noindent\hspace*{0.2cm}\textbf{Order} 0:\\
${\cc_{281}=(\cc_{45}\cc_{62},\ff)_9}$\\
${\cc_{282}=(\cc_{45}\cc_{61},\ff)_9}$\\
${\cc_{283}=(\cc_{45}\cc_{60},\ff)_9}$\\
${\cc_{284}=(\cc_{45}\cc_{59},\ff)_9}$\\
${\cc_{285}=(\cc_{44}\cc_{62},\ff)_9}$\\
${\cc_{286}=(\cc_{44}\cc_{61},\ff)_9}$\\
${\cc_{287}=(\cc_{44}\cc_{60},\ff)_9}$\\
${\cc_{288}=(\cc_{62}\cc_{43},\ff)_9}$\\
${\cc_{289}=(\cc_{61}\cc_{43},\ff)_9}$\\
${\cc_{290}=(\cc_{66}\cc_{41},\ff)_9}$\\
${\cc_{291}=(\cc_{41}\cc_{65},\ff)_9}$\\
${\cc_{292}=(\cc_{41}\cc_{64},\ff)_9}$\\
${\cc_{293}=(\cc_{37}\cc_{72},\ff)_9}$\\
${\cc_{294}=(\cc_{37}\cc_{71},\ff)_9}$\\
\noindent\hspace*{0.2cm}\textbf{Order} 2:\\
${\cc_{295}=(\cc_{280},\ff)_7}$\\
${\cc_{296}=(\cc_{66}\cc_{45},\ff)_9}$\\
${\cc_{297}=(\cc_{45}\cc_{65},\ff)_9}$\\
${\cc_{298}=(\cc_{45}\cc_{64},\ff)_9}$\\
${\cc_{299}=(\cc_{45}\cc_{63},\ff)_9}$\\
${\cc_{300}=(\cc_{45}\cc_{62},\ff)_8}$\\
${\cc_{301}=(\cc_{45}\cc_{61},\ff)_8}$\\
${\cc_{302}=(\cc_{45}\cc_{60},\ff)_8}$\\
${\cc_{303}=(\cc_{45}\cc_{59},\ff)_8}$\\
${\cc_{304}=(\cc_{44}\cc_{66},\ff)_9}$\\
${\cc_{305}=(\cc_{44}\cc_{65},\ff)_9}$\\
${\cc_{306}=(\cc_{44}\cc_{64},\ff)_9}$\\
${\cc_{307}=(\cc_{44}\cc_{63},\ff)_9}$\\
${\cc_{308}=(\cc_{44}\cc_{62},\ff)_8}$\\
${\cc_{309}=(\cc_{44}\cc_{61},\ff)_8}$\\
${\cc_{310}=(\cc_{44}\cc_{60},\ff)_8}$\\
${\cc_{311}=(\cc_{44}\cc_{59},\ff)_8}$\\
${\cc_{312}=(\cc_{66}\cc_{43},\ff)_9}$\\
${\cc_{313}=(\cc_{43}\cc_{65},\ff)_9}$\\
${\cc_{314}=(\cc_{64}\cc_{43},\ff)_9}$\\
${\cc_{315}=(\cc_{63}\cc_{43},\ff)_9}$\\
${\cc_{316}=(\cc_{62}\cc_{43},\ff)_8}$\\
${\cc_{317}=(\cc_{61}\cc_{43},\ff)_8}$\\
\noindent\hspace*{0.2cm}\textbf{Order} 4:\\
${\cc_{318}=(\cc_{280},\ff)_6}$\\
${\cc_{319}=(\cc_{279},\ff)_5}$\\
${\cc_{320}=(\cc_{278},\ff)_5}$\\
${\cc_{321}=(\cc_{277},\ff)_5}$\\
\noindent\hspace*{0.2cm}\textbf{Order} 6:\\
${\cc_{322}=(\cc_{280},\ff)_5}$\\

\noindent\textbf{Degree} 13:\\
\noindent\hspace*{0.2cm}\textbf{Order} 1:\\
${\cc_{323}=(\cc_{66}\cc_{72},\ff)_9}$\\
${\cc_{324}=(\cc_{66}\cc_{71},\ff)_9}$\\
${\cc_{325}=(\cc_{66}\cc_{70},\ff)_9}$\\
${\cc_{326}=(\cc_{66}\cc_{69},\ff)_9}$\\
${\cc_{327}=(\cc_{66}\cc_{68},\ff)_9}$\\
${\cc_{328}=(\cc_{66}\cc_{67},\ff)_9}$\\
${\cc_{329}=(\cc_{66}^2,\ff)_8}$\\
${\cc_{330}=(\cc_{72}\cc_{65},\ff)_9}$\\
${\cc_{331}=(\cc_{71}\cc_{65},\ff)_9}$\\
${\cc_{332}=(\cc_{70}\cc_{65},\ff)_9}$\\
${\cc_{333}=(\cc_{69}\cc_{65},\ff)_9}$\\
${\cc_{334}=(\cc_{68}\cc_{65},\ff)_9}$\\
${\cc_{335}=(\cc_{67}\cc_{65},\ff)_9}$\\
${\cc_{336}=(\cc_{66}\cc_{65},\ff)_8}$\\
${\cc_{337}=(\cc_{65}^2,\ff)_8}$\\
${\cc_{338}=(\cc_{72}\cc_{64},\ff)_9}$\\
${\cc_{339}=(\cc_{71}\cc_{64},\ff)_9}$\\
${\cc_{340}=(\cc_{70}\cc_{64},\ff)_9}$\\
${\cc_{341}=(\cc_{69}\cc_{64},\ff)_9}$\\
${\cc_{342}=(\cc_{68}\cc_{64},\ff)_9}$\\
${\cc_{343}=(\cc_{67}\cc_{64},\ff)_9}$\\
${\cc_{344}=(\cc_{66}\cc_{64},\ff)_8}$\\
${\cc_{345}=(\cc_{64}\cc_{65},\ff)_8}$\\
${\cc_{346}=(\cc_{64}^2,\ff)_8}$\\
${\cc_{347}=(\cc_{72}\cc_{63},\ff)_9}$\\
\noindent\hspace*{0.2cm}\textbf{Order} 3:\\
${\cc_{348}=(\cc_{322},\ff)_6}$\\
${\cc_{349}=(\cc_{72}^2,\ff)_9}$\\
${\cc_{350}=(\cc_{71}\cc_{72},\ff)_9}$\\
${\cc_{351}=(\cc_{71}^2,\ff)_9}$\\
${\cc_{352}=(\cc_{70}\cc_{72},\ff)_9}$\\
${\cc_{353}=(\cc_{70}\cc_{71},\ff)_9}$\\
${\cc_{354}=(\cc_{70}^2,\ff)_9}$\\
${\cc_{355}=(\cc_{69}\cc_{72},\ff)_9}$\\
${\cc_{356}=(\cc_{69}\cc_{71},\ff)_9}$\\
${\cc_{357}=(\cc_{69}\cc_{70},\ff)_9}$\\
\noindent\hspace*{0.2cm}\textbf{Order} 5:\\
${\cc_{358}=(\cc_{322},\ff)_5}$\\

\noindent\textbf{Degree} 14:\\
\noindent\hspace*{0.2cm}\textbf{Order} 0:\\
${\cc_{359}=(\cc_{78}\cc_{89},\ff)_9}$\\
${\cc_{360}=(\cc_{88}\cc_{78},\ff)_9}$\\
${\cc_{361}=(\cc_{78}\cc_{87},\ff)_9}$\\
${\cc_{362}=(\cc_{77}\cc_{89},\ff)_9}$\\
${\cc_{363}=(\cc_{77}\cc_{88},\ff)_9}$\\
${\cc_{364}=(\cc_{89}\cc_{76},\ff)_9}$\\
${\cc_{365}=(\cc_{72}\cc_{97},\ff)_9}$\\
${\cc_{366}=(\cc_{72}\cc_{96},\ff)_9}$\\
${\cc_{367}=(\cc_{72}\cc_{95},\ff)_9}$\\
${\cc_{368}=(\cc_{72}\cc_{94},\ff)_9}$\\
${\cc_{369}=(\cc_{93}\cc_{72},\ff)_9}$\\
${\cc_{370}=(\cc_{92}\cc_{72},\ff)_9}$\\
${\cc_{371}=(\cc_{91}\cc_{72},\ff)_9}$\\
${\cc_{372}=(\cc_{71}\cc_{97},\ff)_9}$\\
${\cc_{373}=(\cc_{71}\cc_{96},\ff)_9}$\\
${\cc_{374}=(\cc_{71}\cc_{95},\ff)_9}$\\
${\cc_{375}=(\cc_{112}\cc_{62},\ff)_9}$\\
\noindent\hspace*{0.2cm}\textbf{Order} 2:\\
${\cc_{376}=(\cc_{78}\cc_{97},\ff)_9}$\\
${\cc_{377}=(\cc_{78}\cc_{96},\ff)_9}$\\
${\cc_{378}=(\cc_{78}\cc_{95},\ff)_9}$\\
${\cc_{379}=(\cc_{78}\cc_{94},\ff)_9}$\\
${\cc_{380}=(\cc_{78}\cc_{93},\ff)_9}$\\
${\cc_{381}=(\cc_{78}\cc_{92},\ff)_9}$\\
${\cc_{382}=(\cc_{78}\cc_{91},\ff)_9}$\\
${\cc_{383}=(\cc_{78}\cc_{90},\ff)_8}$\\
${\cc_{384}=(\cc_{78}\cc_{89},\ff)_8}$\\
${\cc_{385}=(\cc_{88}\cc_{78},\ff)_8}$\\
${\cc_{386}=(\cc_{78}\cc_{87},\ff)_8}$\\
${\cc_{387}=(\cc_{77}\cc_{97},\ff)_9}$\\
${\cc_{388}=(\cc_{66}\cc_{112},\ff)_9}$\\
\noindent\hspace*{0.2cm}\textbf{Order} 4:\\
${\cc_{389}=(\cc_{358},\ff)_5}$\\

\noindent\textbf{Degree} 15:\\
\noindent\hspace*{0.2cm}\textbf{Order} 1:\\
${\cc_{390}=(\cc_{105}^2,\ff)_9}$\\
${\cc_{391}=(\cc_{104}\cc_{105},\ff)_9}$\\
${\cc_{392}=(\cc_{104}^2,\ff)_9}$\\
${\cc_{393}=(\cc_{103}\cc_{105},\ff)_9}$\\
${\cc_{394}=(\cc_{103}\cc_{104},\ff)_9}$\\
${\cc_{395}=(\cc_{103}^2,\ff)_9}$\\
${\cc_{396}=(\cc_{102}\cc_{105},\ff)_9}$\\
${\cc_{397}=(\cc_{102}\cc_{104},\ff)_9}$\\
${\cc_{398}=(\cc_{102}\cc_{103},\ff)_9}$\\
${\cc_{399}=(\cc_{102}^2,\ff)_9}$\\
${\cc_{400}=(\cc_{101}\cc_{105},\ff)_9}$\\
${\cc_{401}=(\cc_{101}\cc_{104},\ff)_9}$\\
${\cc_{402}=(\cc_{101}\cc_{103},\ff)_9}$\\
${\cc_{403}=(\cc_{101}\cc_{102},\ff)_9}$\\
${\cc_{404}=(\cc_{101}^2,\ff)_9}$\\
${\cc_{405}=(\cc_{100}\cc_{105},\ff)_9}$\\
${\cc_{406}=(\cc_{100}\cc_{104},\ff)_9}$\\
${\cc_{407}=(\cc_{100}\cc_{103},\ff)_9}$\\
${\cc_{408}=(\cc_{100}\cc_{102},\ff)_9}$\\
${\cc_{409}=(\cc_{100}\cc_{101},\ff)_9}$\\
${\cc_{410}=(\cc_{100}^2,\ff)_9}$\\
${\cc_{411}=(\cc_{99}\cc_{105},\ff)_9}$\\
${\cc_{412}=(\cc_{99}\cc_{104},\ff)_9}$\\
${\cc_{413}=(\cc_{99}\cc_{103},\ff)_9}$\\
${\cc_{414}=(\cc_{99}\cc_{102},\ff)_9}$\\
${\cc_{415}=(\cc_{112}\cc_{97},\ff)_9}$\\
\noindent\hspace*{0.2cm}\textbf{Order} 3:\\
${\cc_{416}=(\cc_{112}\cc_{105},\ff)_9}$\\

\noindent\textbf{Degree} 16:\\
\noindent\hspace*{0.2cm}\textbf{Order} 0:\\
${\cc_{417}=(\cc_{133}\cc_{112},\ff)_9}$\\
${\cc_{418}=(\cc_{132}\cc_{112},\ff)_9}$\\
${\cc_{419}=(\cc_{112}\cc_{131},\ff)_9}$\\
${\cc_{420}=(\cc_{112}\cc_{130},\ff)_9}$\\
${\cc_{421}=(\cc_{112}\cc_{129},\ff)_9}$\\
${\cc_{422}=(\cc_{112}\cc_{128},\ff)_9}$\\
${\cc_{423}=(\cc_{112}\cc_{127},\ff)_9}$\\
${\cc_{424}=(\cc_{112}\cc_{126},\ff)_9}$\\
${\cc_{425}=(\cc_{132}\cc_{111},\ff)_9}$\\
${\cc_{426}=(\cc_{111}\cc_{131},\ff)_9}$\\
${\cc_{427}=(\cc_{111}\cc_{130},\ff)_9}$\\
${\cc_{428}=(\cc_{111}\cc_{129},\ff)_9}$\\
${\cc_{429}=(\cc_{111}\cc_{128},\ff)_9}$\\
${\cc_{430}=(\cc_{111}\cc_{127},\ff)_9}$\\
${\cc_{431}=(\cc_{111}\cc_{126},\ff)_9}$\\
${\cc_{432}=(\cc_{110}\cc_{132},\ff)_9}$\\
${\cc_{433}=(\cc_{110}\cc_{131},\ff)_9}$\\
${\cc_{434}=(\cc_{110}\cc_{130},\ff)_9}$\\
${\cc_{435}=(\cc_{143}\cc_{105},\ff)_9}$\\
${\cc_{436}=(\cc_{105}\cc_{142},\ff)_9}$\\
${\cc_{437}=(\cc_{90}\cc_{157},\ff)_9}$\\
\noindent\hspace*{0.2cm}\textbf{Order} 2:\\
${\cc_{438}=(\cc_{143}\cc_{112},\ff)_9}$\\
${\cc_{439}=(\cc_{112}\cc_{142},\ff)_9}$\\
${\cc_{440}=(\cc_{112}\cc_{141},\ff)_9}$\\

\noindent\textbf{Degree} 17:\\
\noindent\hspace*{0.2cm}\textbf{Order} 1:\\
${\cc_{441}=(\cc_{143}\cc_{153},\ff)_9}$\\
${\cc_{442}=(\cc_{143}\cc_{152},\ff)_9}$\\
${\cc_{443}=(\cc_{143}\cc_{151},\ff)_9}$\\
${\cc_{444}=(\cc_{143}\cc_{150},\ff)_9}$\\
${\cc_{445}=(\cc_{143}\cc_{149},\ff)_9}$\\
${\cc_{446}=(\cc_{143}\cc_{148},\ff)_9}$\\
${\cc_{447}=(\cc_{143}\cc_{147},\ff)_9}$\\

\noindent\textbf{Degree} 18:\\
\noindent\hspace*{0.2cm}\textbf{Order} 0:\\
${\cc_{448}=(\cc_{157}\cc_{168},\ff)_9}$\\
${\cc_{449}=(\cc_{167}\cc_{157},\ff)_9}$\\
${\cc_{450}=(\cc_{166}\cc_{157},\ff)_9}$\\
${\cc_{451}=(\cc_{165}\cc_{157},\ff)_9}$\\
${\cc_{452}=(\cc_{157}\cc_{164},\ff)_9}$\\
${\cc_{453}=(\cc_{157}\cc_{163},\ff)_9}$\\
${\cc_{454}=(\cc_{157}\cc_{161},\ff)_9}$\\
${\cc_{455}=(\cc_{157}\cc_{160},\ff)_9}$\\
${\cc_{456}=(\cc_{156}\cc_{168},\ff)_9}$\\
${\cc_{457}=(\cc_{156}\cc_{167},\ff)_9}$\\
${\cc_{458}=(\cc_{166}\cc_{156},\ff)_9}$\\
${\cc_{459}=(\cc_{165}\cc_{156},\ff)_9}$\\
${\cc_{460}=(\cc_{156}\cc_{164},\ff)_9}$\\
${\cc_{461}=(\cc_{156}\cc_{161},\ff)_9}$\\
${\cc_{462}=(\cc_{156}\cc_{160},\ff)_9}$\\
${\cc_{463}=(\cc_{182}\cc_{153},\ff)_9}$\\
${\cc_{464}=(\cc_{181}\cc_{153},\ff)_9}$\\
${\cc_{465}=(\cc_{180}\cc_{153},\ff)_9}$\\
${\cc_{466}=(\cc_{179}\cc_{153},\ff)_9}$\\
${\cc_{467}=(\cc_{178}\cc_{153},\ff)_9}$\\
${\cc_{468}=(\cc_{177}\cc_{153},\ff)_9}$\\
${\cc_{469}=(\cc_{176}\cc_{153},\ff)_9}$\\
${\cc_{470}=(\cc_{174}\cc_{153},\ff)_9}$\\
${\cc_{471}=(\cc_{173}\cc_{153},\ff)_9}$\\
${\cc_{472}=(\cc_{172}\cc_{153},\ff)_9}$\\

\noindent\textbf{Degree} 19:\\
\noindent\hspace*{0.2cm}\textbf{Order} 1:\\
${\cc_{473}=(\cc_{192}^2,\ff)_9}$\\

\noindent\textbf{Degree} 20:\\
\noindent\hspace*{0.2cm}\textbf{Order} 0:\\
${\cc_{474}=(\cc_{220}\cc_{199},\ff)_9}$\\
${\cc_{475}=(\cc_{199}\cc_{219},\ff)_9}$\\

\noindent\textbf{Degree} 22:\\
\noindent\hspace*{0.2cm}\textbf{Order} 0:\\
${\cc_{476}=(\cc_{238}\cc_{272},\ff)_9}$\\
  \end{multicols}
\end{tiny}

\pagebreak
\section{A minimal basis for $\cov{\Sn{10}}$}
\label{sec:mathbfcov_10}

\begin{tiny}
  \begin{multicols}{6}
\noindent\textbf{Degree} 1:\\
\noindent\hspace*{0.2cm}\textbf{Order} 10:\\
${\cc_{1}={\mathfrak f}}$\\

\noindent\textbf{Degree} 2:\\
\noindent\hspace*{0.2cm}\textbf{Order} 0:\\
${\cc_{2}=({\mathfrak f},{\mathfrak f})_{10}}$\\
\noindent\hspace*{0.2cm}\textbf{Order} 4:\\
${\cc_{3}=({\mathfrak f},{\mathfrak f})_8}$\\
\noindent\hspace*{0.2cm}\textbf{Order} 8:\\
${\cc_{4}=({\mathfrak f},{\mathfrak f})_6}$\\
\noindent\hspace*{0.2cm}\textbf{Order} 12:\\
${\cc_{5}=(\ff,\ff)_4}$\\
\noindent\hspace*{0.2cm}\textbf{Order} 16:\\
${\cc_{6}=(\ff,\ff)_2}$\\

\noindent\textbf{Degree} 3:\\
\noindent\hspace*{0.2cm}\textbf{Order} 2:\\
${\cc_{7}=(\cc_{4},\ff)_8}$\\
\noindent\hspace*{0.2cm}\textbf{Order} 6:\\
${\cc_{8}=(\cc_{3},\ff)_4}$\\
${\cc_{9}=(\cc_{5},\ff)_8}$\\
\noindent\hspace*{0.2cm}\textbf{Order} 8:\\
${\cc_{10}=(\cc_{6},\ff)_9}$\\
\noindent\hspace*{0.2cm}\textbf{Order} 10:\\
${\cc_{11}=(\cc_{6},\ff)_8}$\\
\noindent\hspace*{0.2cm}\textbf{Order} 12:\\
${\cc_{12}=(\cc_{6},\ff)_7}$\\
${\cc_{13}=(\cc_{5},\ff)_5}$\\
\noindent\hspace*{0.2cm}\textbf{Order} 14:\\
${\cc_{14}=(\cc_{6},\ff)_6}$\\
\noindent\hspace*{0.2cm}\textbf{Order} 16:\\
${\cc_{15}=(\cc_{6},\ff)_5}$\\
\noindent\hspace*{0.2cm}\textbf{Order} 18:\\
${\cc_{16}=(\cc_{6},\ff)_4}$\\
\noindent\hspace*{0.2cm}\textbf{Order} 20:\\
${\cc_{17}=(\cc_{6},\ff)_3}$\\
\noindent\hspace*{0.2cm}\textbf{Order} 24:\\
${\cc_{18}=(\cc_{6},\ff)_1}$\\

\noindent\textbf{Degree} 4:\\
\noindent\hspace*{0.2cm}\textbf{Order} 0:\\
${\cc_{19}=(\cc_{3},\cc_{3})_4}$\\
\noindent\hspace*{0.2cm}\textbf{Order} 4:\\
${\cc_{20}=(\cc_{14},\ff)_{10}}$\\
${\cc_{21}=(\cc_{13},\ff)_9}$\\
${\cc_{22}=(\cc_{12},\ff)_9}$\\
\noindent\hspace*{0.2cm}\textbf{Order} 6:\\
${\cc_{23}=(\cc_{15},\ff)_{10}}$\\
\noindent\hspace*{0.2cm}\textbf{Order} 8:\\
${\cc_{24}=(\cc_{16},\ff)_{10}}$\\
${\cc_{25}=(\cc_{15},\ff)_9}$\\
${\cc_{26}=(\cc_{14},\ff)_8}$\\
\noindent\hspace*{0.2cm}\textbf{Order} 10:\\
${\cc_{27}=(\cc_{17},\ff)_{10}}$\\
${\cc_{28}=(\cc_{16},\ff)_9}$\\
${\cc_{29}=(\cc_{15},\ff)_8}$\\
\noindent\hspace*{0.2cm}\textbf{Order} 12:\\
${\cc_{30}=(\cc_{17},\ff)_9}$\\
${\cc_{31}=(\cc_{16},\ff)_8}$\\
\noindent\hspace*{0.2cm}\textbf{Order} 14:\\
${\cc_{32}=(\cc_{18},\ff)_{10}}$\\
${\cc_{33}=(\cc_{17},\ff)_8}$\\
${\cc_{34}=(\cc_{16},\ff)_7}$\\
\noindent\hspace*{0.2cm}\textbf{Order} 16:\\
${\cc_{35}=(\cc_{18},\ff)_9}$\\
\noindent\hspace*{0.2cm}\textbf{Order} 18:\\
${\cc_{36}=(\cc_{18},\ff)_8}$\\
${\cc_{37}=(\cc_{17},\ff)_6}$\\
\noindent\hspace*{0.2cm}\textbf{Order} 20:\\
${\cc_{38}=(\cc_{18},\ff)_7}$\\
\noindent\hspace*{0.2cm}\textbf{Order} 22:\\
${\cc_{39}=(\cc_{18},\ff)_6}$\\
\noindent\hspace*{0.2cm}\textbf{Order} 26:\\
${\cc_{40}=(\cc_{18},\ff)_4}$\\

\noindent\textbf{Degree} 5:\\
\noindent\hspace*{0.2cm}\textbf{Order} 2:\\
${\cc_{41}=(\cc_{31},\ff)_{10}}$\\
${\cc_{42}=(\cc_{30},\ff)_{10}}$\\
${\cc_{43}=(\cc_{29},\ff)_9}$\\
\noindent\hspace*{0.2cm}\textbf{Order} 4:\\
${\cc_{44}=(\cc_{34},\ff)_{10}}$\\
${\cc_{45}=(\cc_{33},\ff)_{10}}$\\
${\cc_{46}=(\cc_{32},\ff)_{10}}$\\
\noindent\hspace*{0.2cm}\textbf{Order} 6:\\
${\cc_{47}=(\cc_{35},\ff)_{10}}$\\
${\cc_{48}=(\cc_{34},\ff)_9}$\\
${\cc_{49}=(\cc_{33},\ff)_9}$\\
${\cc_{50}=(\cc_{31},\ff)_8}$\\
\noindent\hspace*{0.2cm}\textbf{Order} 8:\\
${\cc_{51}=(\cc_{37},\ff)_{10}}$\\
${\cc_{52}=(\cc_{36},\ff)_{10}}$\\
${\cc_{53}=(\cc_{35},\ff)_9}$\\
${\cc_{54}=(\cc_{34},\ff)_8}$\\
${\cc_{55}=(\cc_{33},\ff)_8}$\\
\noindent\hspace*{0.2cm}\textbf{Order} 10:\\
${\cc_{56}=(\cc_{38},\ff)_{10}}$\\
${\cc_{57}=(\cc_{37},\ff)_9}$\\
${\cc_{58}=(\cc_{36},\ff)_9}$\\
${\cc_{59}=(\cc_{34},\ff)_7}$\\
\noindent\hspace*{0.2cm}\textbf{Order} 12:\\
${\cc_{60}=(\cc_{39},\ff)_{10}}$\\
${\cc_{61}=(\cc_{38},\ff)_9}$\\
${\cc_{62}=(\cc_{37},\ff)_8}$\\
${\cc_{63}=(\cc_{36},\ff)_8}$\\
${\cc_{64}=(\cc_{34},\ff)_6}$\\
\noindent\hspace*{0.2cm}\textbf{Order} 14:\\
${\cc_{65}=(\cc_{39},\ff)_9}$\\
${\cc_{66}=(\cc_{38},\ff)_8}$\\
\noindent\hspace*{0.2cm}\textbf{Order} 16:\\
${\cc_{67}=(\cc_{40},\ff)_{10}}$\\
${\cc_{68}=(\cc_{39},\ff)_8}$\\
${\cc_{69}=(\cc_{38},\ff)_7}$\\
${\cc_{70}=(\cc_{37},\ff)_6}$\\
\noindent\hspace*{0.2cm}\textbf{Order} 20:\\
${\cc_{71}=(\cc_{40},\ff)_8}$\\
${\cc_{72}=(\cc_{39},\ff)_6}$\\

\noindent\textbf{Degree} 6:\\
\noindent\hspace*{0.2cm}\textbf{Order} 0:\\
${\cc_{73}=(\cc_{8},\cc_{8})_6}$\\
${\cc_{74}=(\cc_{7},\cc_{7})_2}$\\
${\cc_{75}=(\cc_{56},\ff)_{10}}$\\
${\cc_{76}=(\cc_{3}\cc_{9},\ff)_{10}}$\\
\noindent\hspace*{0.2cm}\textbf{Order} 2:\\
${\cc_{77}=(\cc_{64},\ff)_{10}}$\\
${\cc_{78}=(\cc_{63},\ff)_{10}}$\\
\noindent\hspace*{0.2cm}\textbf{Order} 4:\\
${\cc_{79}=(\cc_{66},\ff)_{10}}$\\
${\cc_{80}=(\cc_{65},\ff)_{10}}$\\
${\cc_{81}=(\cc_{64},\ff)_9}$\\
${\cc_{82}=(\cc_{63},\ff)_9}$\\
${\cc_{83}=(\cc_{62},\ff)_9}$\\
\noindent\hspace*{0.2cm}\textbf{Order} 6:\\
${\cc_{84}=(\cc_{70},\ff)_{10}}$\\
${\cc_{85}=(\cc_{69},\ff)_{10}}$\\
${\cc_{86}=(\cc_{68},\ff)_{10}}$\\
${\cc_{87}=(\cc_{67},\ff)_{10}}$\\
${\cc_{88}=(\cc_{66},\ff)_9}$\\
${\cc_{89}=(\cc_{65},\ff)_9}$\\
${\cc_{90}=(\cc_{64},\ff)_8}$\\
${\cc_{91}=(\cc_{63},\ff)_8}$\\
\noindent\hspace*{0.2cm}\textbf{Order} 8:\\
${\cc_{92}=(\cc_{70},\ff)_9}$\\
${\cc_{93}=(\cc_{69},\ff)_9}$\\
${\cc_{94}=(\cc_{68},\ff)_9}$\\
${\cc_{95}=(\cc_{67},\ff)_9}$\\
${\cc_{96}=(\cc_{66},\ff)_8}$\\
${\cc_{97}=(\cc_{65},\ff)_8}$\\
\noindent\hspace*{0.2cm}\textbf{Order} 10:\\
${\cc_{98}=(\cc_{72},\ff)_{10}}$\\
${\cc_{99}=(\cc_{71},\ff)_{10}}$\\
${\cc_{100}=(\cc_{70},\ff)_8}$\\
${\cc_{101}=(\cc_{69},\ff)_8}$\\
${\cc_{102}=(\cc_{68},\ff)_8}$\\
${\cc_{103}=(\cc_{67},\ff)_8}$\\
${\cc_{104}=(\cc_{66},\ff)_7}$\\
${\cc_{105}=(\cc_{65},\ff)_7}$\\
\noindent\hspace*{0.2cm}\textbf{Order} 12:\\
${\cc_{106}=(\cc_{72},\ff)_9}$\\
${\cc_{107}=(\cc_{71},\ff)_9}$\\
\noindent\hspace*{0.2cm}\textbf{Order} 14:\\
${\cc_{108}=(\cc_{72},\ff)_8}$\\
${\cc_{109}=(\cc_{71},\ff)_8}$\\
${\cc_{110}=(\cc_{70},\ff)_6}$\\
${\cc_{111}=(\cc_{69},\ff)_6}$\\
\noindent\hspace*{0.2cm}\textbf{Order} 18:\\
${\cc_{112}=(\cc_{72},\ff)_6}$\\

\noindent\textbf{Degree} 7:\\
\noindent\hspace*{0.2cm}\textbf{Order} 2:\\
${\cc_{113}=(\cc_{107},\ff)_{10}}$\\
${\cc_{114}=(\cc_{106},\ff)_{10}}$\\
${\cc_{115}=(\cc_{105},\ff)_9}$\\
${\cc_{116}=(\cc_{104},\ff)_9}$\\
${\cc_{117}=(\cc_{103},\ff)_9}$\\
${\cc_{118}=(\cc_{102},\ff)_9}$\\
${\cc_{119}=(\cc_{101},\ff)_9}$\\
\noindent\hspace*{0.2cm}\textbf{Order} 4:\\
${\cc_{120}=(\cc_{111},\ff)_{10}}$\\
${\cc_{121}=(\cc_{110},\ff)_{10}}$\\
${\cc_{122}=(\cc_{109},\ff)_{10}}$\\
${\cc_{123}=(\cc_{108},\ff)_{10}}$\\
${\cc_{124}=(\cc_{107},\ff)_9}$\\
${\cc_{125}=(\cc_{106},\ff)_9}$\\
${\cc_{126}=(\cc_{105},\ff)_8}$\\
${\cc_{127}=(\cc_{104},\ff)_8}$\\
${\cc_{128}=(\cc_{103},\ff)_8}$\\
${\cc_{129}=(\cc_{102},\ff)_8}$\\
\noindent\hspace*{0.2cm}\textbf{Order} 6:\\
${\cc_{130}=(\cc_{111},\ff)_9}$\\
${\cc_{131}=(\cc_{110},\ff)_9}$\\
${\cc_{132}=(\cc_{109},\ff)_9}$\\
${\cc_{133}=(\cc_{108},\ff)_9}$\\
${\cc_{134}=(\cc_{107},\ff)_8}$\\
${\cc_{135}=(\cc_{106},\ff)_8}$\\
${\cc_{136}=(\cc_{105},\ff)_7}$\\
${\cc_{137}=(\cc_{104},\ff)_7}$\\
\noindent\hspace*{0.2cm}\textbf{Order} 8:\\
${\cc_{138}=(\cc_{112},\ff)_{10}}$\\
${\cc_{139}=(\cc_{111},\ff)_8}$\\
${\cc_{140}=(\cc_{110},\ff)_8}$\\
${\cc_{141}=(\cc_{109},\ff)_8}$\\
${\cc_{142}=(\cc_{108},\ff)_8}$\\
${\cc_{143}=(\cc_{107},\ff)_7}$\\
${\cc_{144}=(\cc_{106},\ff)_7}$\\
${\cc_{145}=(\cc_{105},\ff)_6}$\\
${\cc_{146}=(\cc_{104},\ff)_6}$\\
${\cc_{147}=(\cc_{103},\ff)_6}$\\
${\cc_{148}=(\cc_{102},\ff)_6}$\\
${\cc_{149}=(\cc_{101},\ff)_6}$\\
\noindent\hspace*{0.2cm}\textbf{Order} 10:\\
${\cc_{150}=(\cc_{112},\ff)_9}$\\
${\cc_{151}=(\cc_{111},\ff)_7}$\\
\noindent\hspace*{0.2cm}\textbf{Order} 12:\\
${\cc_{152}=(\cc_{112},\ff)_8}$\\
${\cc_{153}=(\cc_{111},\ff)_6}$\\
${\cc_{154}=(\cc_{110},\ff)_6}$\\
${\cc_{155}=(\cc_{109},\ff)_6}$\\
\noindent\hspace*{0.2cm}\textbf{Order} 16:\\
${\cc_{156}=(\cc_{112},\ff)_6}$\\

\noindent\textbf{Degree} 8:\\
\noindent\hspace*{0.2cm}\textbf{Order} 0:\\
${\cc_{157}=(\cc_{22}\cc_{9},\ff)_{10}}$\\
${\cc_{158}=(\cc_{9}\cc_{21},\ff)_{10}}$\\
${\cc_{159}=(\cc_{9}\cc_{20},\ff)_{10}}$\\
${\cc_{160}=(\cc_{22}\cc_{8},\ff)_{10}}$\\
${\cc_{161}=(\cc_{26}\cc_{7},\ff)_{10}}$\\
\noindent\hspace*{0.2cm}\textbf{Order} 2:\\
${\cc_{162}=(\cc_{155},\ff)_{10}}$\\
${\cc_{163}=(\cc_{154},\ff)_{10}}$\\
${\cc_{164}=(\cc_{153},\ff)_{10}}$\\
${\cc_{165}=(\cc_{152},\ff)_{10}}$\\
${\cc_{166}=(\cc_{151},\ff)_9}$\\
${\cc_{167}=(\cc_{150},\ff)_9}$\\
${\cc_{168}=(\cc_{149},\ff)_8}$\\
${\cc_{169}=(\cc_{148},\ff)_8}$\\
\noindent\hspace*{0.2cm}\textbf{Order} 4:\\
${\cc_{170}=(\cc_{155},\ff)_9}$\\
${\cc_{171}=(\cc_{154},\ff)_9}$\\
${\cc_{172}=(\cc_{153},\ff)_9}$\\
${\cc_{173}=(\cc_{152},\ff)_9}$\\
${\cc_{174}=(\cc_{151},\ff)_8}$\\
${\cc_{175}=(\cc_{150},\ff)_8}$\\
${\cc_{176}=(\cc_{149},\ff)_7}$\\
${\cc_{177}=(\cc_{148},\ff)_7}$\\
${\cc_{178}=(\cc_{147},\ff)_7}$\\
${\cc_{179}=(\cc_{146},\ff)_7}$\\
${\cc_{180}=(\cc_{145},\ff)_7}$\\
\noindent\hspace*{0.2cm}\textbf{Order} 6:\\
${\cc_{181}=(\cc_{156},\ff)_{10}}$\\
${\cc_{182}=(\cc_{155},\ff)_8}$\\
${\cc_{183}=(\cc_{154},\ff)_8}$\\
${\cc_{184}=(\cc_{153},\ff)_8}$\\
${\cc_{185}=(\cc_{152},\ff)_8}$\\
${\cc_{186}=(\cc_{151},\ff)_7}$\\
${\cc_{187}=(\cc_{150},\ff)_7}$\\
${\cc_{188}=(\cc_{149},\ff)_6}$\\
${\cc_{189}=(\cc_{148},\ff)_6}$\\
${\cc_{190}=(\cc_{147},\ff)_6}$\\
${\cc_{191}=(\cc_{146},\ff)_6}$\\
${\cc_{192}=(\cc_{145},\ff)_6}$\\
${\cc_{193}=(\cc_{144},\ff)_6}$\\
${\cc_{194}=(\cc_{143},\ff)_6}$\\
${\cc_{195}=(\cc_{142},\ff)_6}$\\
\noindent\hspace*{0.2cm}\textbf{Order} 8:\\
${\cc_{196}=(\cc_{156},\ff)_9}$\\
${\cc_{197}=(\cc_{155},\ff)_7}$\\
${\cc_{198}=(\cc_{154},\ff)_7}$\\
${\cc_{199}=(\cc_{153},\ff)_7}$\\
\noindent\hspace*{0.2cm}\textbf{Order} 10:\\
${\cc_{200}=(\cc_{156},\ff)_8}$\\
${\cc_{201}=(\cc_{155},\ff)_6}$\\
${\cc_{202}=(\cc_{154},\ff)_6}$\\
${\cc_{203}=(\cc_{153},\ff)_6}$\\
${\cc_{204}=(\cc_{152},\ff)_6}$\\
${\cc_{205}=(\cc_{151},\ff)_5}$\\
${\cc_{206}=(\cc_{150},\ff)_5}$\\
\noindent\hspace*{0.2cm}\textbf{Order} 14:\\
${\cc_{207}=(\cc_{156},\ff)_6}$\\

\noindent\textbf{Degree} 9:\\
\noindent\hspace*{0.2cm}\textbf{Order} 0:\\
${\cc_{208}=(\cc_{204},\ff)_{10}}$\\
${\cc_{209}=(\cc_{203},\ff)_{10}}$\\
${\cc_{210}=(\cc_{202},\ff)_{10}}$\\
${\cc_{211}=(\cc_{201},\ff)_{10}}$\\
${\cc_{212}=(\cc_{200},\ff)_{10}}$\\
\noindent\hspace*{0.2cm}\textbf{Order} 2:\\
${\cc_{213}=(\cc_{206},\ff)_9}$\\
${\cc_{214}=(\cc_{205},\ff)_9}$\\
${\cc_{215}=(\cc_{204},\ff)_9}$\\
${\cc_{216}=(\cc_{203},\ff)_9}$\\
${\cc_{217}=(\cc_{202},\ff)_9}$\\
${\cc_{218}=(\cc_{201},\ff)_9}$\\
${\cc_{219}=(\cc_{200},\ff)_9}$\\
${\cc_{220}=(\cc_{23}^2,\ff)_{10}}$\\
${\cc_{221}=(\cc_{22}\cc_{26},\ff)_{10}}$\\
${\cc_{222}=(\cc_{22}\cc_{25},\ff)_{10}}$\\
${\cc_{223}=(\cc_{22}\cc_{24},\ff)_{10}}$\\
${\cc_{224}=(\cc_{22}\cc_{23},\ff)_9}$\\
${\cc_{225}=(\cc_{22}^2,\ff)_8}$\\
\noindent\hspace*{0.2cm}\textbf{Order} 4:\\
${\cc_{226}=(\cc_{207},\ff)_{10}}$\\
${\cc_{227}=(\cc_{206},\ff)_8}$\\
${\cc_{228}=(\cc_{205},\ff)_8}$\\
${\cc_{229}=(\cc_{204},\ff)_8}$\\
${\cc_{230}=(\cc_{203},\ff)_8}$\\
${\cc_{231}=(\cc_{202},\ff)_8}$\\
${\cc_{232}=(\cc_{201},\ff)_8}$\\
${\cc_{233}=(\cc_{200},\ff)_8}$\\
${\cc_{234}=(\cc_{199},\ff)_7}$\\
${\cc_{235}=(\cc_{198},\ff)_7}$\\
${\cc_{236}=(\cc_{197},\ff)_7}$\\
${\cc_{237}=(\cc_{196},\ff)_7}$\\
${\cc_{238}=(\cc_{195},\ff)_6}$\\
${\cc_{239}=(\cc_{194},\ff)_6}$\\
${\cc_{240}=(\cc_{193},\ff)_6}$\\
${\cc_{241}=(\cc_{192},\ff)_6}$\\
${\cc_{242}=(\cc_{191},\ff)_6}$\\
${\cc_{243}=(\cc_{190},\ff)_6}$\\
${\cc_{244}=(\cc_{189},\ff)_6}$\\
\noindent\hspace*{0.2cm}\textbf{Order} 6:\\
${\cc_{245}=(\cc_{207},\ff)_9}$\\
${\cc_{246}=(\cc_{206},\ff)_7}$\\
${\cc_{247}=(\cc_{205},\ff)_7}$\\
${\cc_{248}=(\cc_{204},\ff)_7}$\\
${\cc_{249}=(\cc_{203},\ff)_7}$\\
${\cc_{250}=(\cc_{202},\ff)_7}$\\
${\cc_{251}=(\cc_{201},\ff)_7}$\\
${\cc_{252}=(\cc_{200},\ff)_7}$\\
\noindent\hspace*{0.2cm}\textbf{Order} 8:\\
${\cc_{253}=(\cc_{207},\ff)_8}$\\
${\cc_{254}=(\cc_{206},\ff)_6}$\\
${\cc_{255}=(\cc_{205},\ff)_6}$\\
${\cc_{256}=(\cc_{204},\ff)_6}$\\
${\cc_{257}=(\cc_{203},\ff)_6}$\\
${\cc_{258}=(\cc_{202},\ff)_6}$\\
${\cc_{259}=(\cc_{201},\ff)_6}$\\
\noindent\hspace*{0.2cm}\textbf{Order} 12:\\
${\cc_{260}=(\cc_{206},\ff)_4}$\\

\noindent\textbf{Degree} 10:\\
\noindent\hspace*{0.2cm}\textbf{Order} 0:\\
${\cc_{261}=(\cc_{26}\cc_{43},\ff)_{10}}$\\
${\cc_{262}=(\cc_{26}\cc_{42},\ff)_{10}}$\\
${\cc_{263}=(\cc_{26}\cc_{41},\ff)_{10}}$\\
${\cc_{264}=(\cc_{25}\cc_{43},\ff)_{10}}$\\
${\cc_{265}=(\cc_{25}\cc_{42},\ff)_{10}}$\\
${\cc_{266}=(\cc_{23}\cc_{46},\ff)_{10}}$\\
${\cc_{267}=(\cc_{23}\cc_{45},\ff)_{10}}$\\
${\cc_{268}=(\cc_{22}\cc_{50},\ff)_{10}}$\\
\noindent\hspace*{0.2cm}\textbf{Order} 2:\\
${\cc_{269}=(\cc_{260},\ff)_{10}}$\\
${\cc_{270}=(\cc_{259},\ff)_8}$\\
${\cc_{271}=(\cc_{258},\ff)_8}$\\
${\cc_{272}=(\cc_{257},\ff)_8}$\\
${\cc_{273}=(\cc_{256},\ff)_8}$\\
${\cc_{274}=(\cc_{255},\ff)_8}$\\
${\cc_{275}=(\cc_{254},\ff)_8}$\\
${\cc_{276}=(\cc_{253},\ff)_8}$\\
${\cc_{277}=(\cc_{46}\cc_{26},\ff)_{10}}$\\
${\cc_{278}=(\cc_{45}\cc_{26},\ff)_{10}}$\\
${\cc_{279}=(\cc_{44}\cc_{26},\ff)_{10}}$\\
${\cc_{280}=(\cc_{26}\cc_{43},\ff)_9}$\\
${\cc_{281}=(\cc_{26}\cc_{42},\ff)_9}$\\
${\cc_{282}=(\cc_{26}\cc_{41},\ff)_9}$\\
${\cc_{283}=(\cc_{46}\cc_{25},\ff)_{10}}$\\
${\cc_{284}=(\cc_{45}\cc_{25},\ff)_{10}}$\\
${\cc_{285}=(\cc_{44}\cc_{25},\ff)_{10}}$\\
${\cc_{286}=(\cc_{25}\cc_{43},\ff)_9}$\\
${\cc_{287}=(\cc_{25}\cc_{42},\ff)_9}$\\
${\cc_{288}=(\cc_{25}\cc_{41},\ff)_9}$\\
\noindent\hspace*{0.2cm}\textbf{Order} 4:\\
${\cc_{289}=(\cc_{260},\ff)_9}$\\
${\cc_{290}=(\cc_{259},\ff)_7}$\\
${\cc_{291}=(\cc_{258},\ff)_7}$\\
${\cc_{292}=(\cc_{257},\ff)_7}$\\
${\cc_{293}=(\cc_{256},\ff)_7}$\\
${\cc_{294}=(\cc_{255},\ff)_7}$\\
${\cc_{295}=(\cc_{254},\ff)_7}$\\
${\cc_{296}=(\cc_{253},\ff)_7}$\\
${\cc_{297}=(\cc_{252},\ff)_6}$\\
${\cc_{298}=(\cc_{251},\ff)_6}$\\
${\cc_{299}=(\cc_{250},\ff)_6}$\\
${\cc_{300}=(\cc_{249},\ff)_6}$\\
${\cc_{301}=(\cc_{248},\ff)_6}$\\
\noindent\hspace*{0.2cm}\textbf{Order} 6:\\
${\cc_{302}=(\cc_{260},\ff)_8}$\\
${\cc_{303}=(\cc_{259},\ff)_6}$\\
${\cc_{304}=(\cc_{258},\ff)_6}$\\
${\cc_{305}=(\cc_{257},\ff)_6}$\\
${\cc_{306}=(\cc_{256},\ff)_6}$\\
${\cc_{307}=(\cc_{255},\ff)_6}$\\
${\cc_{308}=(\cc_{254},\ff)_6}$\\
${\cc_{309}=(\cc_{253},\ff)_6}$\\
${\cc_{310}=(\cc_{252},\ff)_5}$\\
${\cc_{311}=(\cc_{251},\ff)_5}$\\
${\cc_{312}=(\cc_{250},\ff)_5}$\\
${\cc_{313}=(\cc_{249},\ff)_5}$\\
${\cc_{314}=(\cc_{248},\ff)_5}$\\
\noindent\hspace*{0.2cm}\textbf{Order} 10:\\
${\cc_{315}=(\cc_{260},\ff)_6}$\\

\noindent\textbf{Degree} 11:\\
\noindent\hspace*{0.2cm}\textbf{Order} 0:\\
${\cc_{316}=(\cc_{315},\ff)_{10}}$\\
${\cc_{317}=(\cc_{46}\cc_{50},\ff)_{10}}$\\
${\cc_{318}=(\cc_{46}\cc_{49},\ff)_{10}}$\\
${\cc_{319}=(\cc_{46}\cc_{48},\ff)_{10}}$\\
${\cc_{320}=(\cc_{46}\cc_{47},\ff)_{10}}$\\
${\cc_{321}=(\cc_{45}\cc_{50},\ff)_{10}}$\\
${\cc_{322}=(\cc_{45}\cc_{49},\ff)_{10}}$\\
${\cc_{323}=(\cc_{45}\cc_{48},\ff)_{10}}$\\
\noindent\hspace*{0.2cm}\textbf{Order} 2:\\
${\cc_{324}=(\cc_{315},\ff)_9}$\\
${\cc_{325}=(\cc_{50}^2,\ff)_{10}}$\\
${\cc_{326}=(\cc_{49}\cc_{50},\ff)_{10}}$\\
${\cc_{327}=(\cc_{49}^2,\ff)_{10}}$\\
${\cc_{328}=(\cc_{48}\cc_{50},\ff)_{10}}$\\
${\cc_{329}=(\cc_{48}\cc_{49},\ff)_{10}}$\\
${\cc_{330}=(\cc_{48}^2,\ff)_{10}}$\\
${\cc_{331}=(\cc_{47}\cc_{50},\ff)_{10}}$\\
${\cc_{332}=(\cc_{47}\cc_{49},\ff)_{10}}$\\
${\cc_{333}=(\cc_{47}\cc_{48},\ff)_{10}}$\\
${\cc_{334}=(\cc_{47}^2,\ff)_{10}}$\\
${\cc_{335}=(\cc_{55}\cc_{46},\ff)_{10}}$\\
${\cc_{336}=(\cc_{46}\cc_{54},\ff)_{10}}$\\
${\cc_{337}=(\cc_{46}\cc_{53},\ff)_{10}}$\\
${\cc_{338}=(\cc_{46}\cc_{52},\ff)_{10}}$\\
${\cc_{339}=(\cc_{46}\cc_{51},\ff)_{10}}$\\
${\cc_{340}=(\cc_{46}\cc_{50},\ff)_9}$\\
${\cc_{341}=(\cc_{46}\cc_{49},\ff)_9}$\\
\noindent\hspace*{0.2cm}\textbf{Order} 4:\\
${\cc_{342}=(\cc_{315},\ff)_8}$\\
${\cc_{343}=(\cc_{314},\ff)_6}$\\
${\cc_{344}=(\cc_{313},\ff)_6}$\\
${\cc_{345}=(\cc_{312},\ff)_6}$\\
${\cc_{346}=(\cc_{311},\ff)_6}$\\
${\cc_{347}=(\cc_{310},\ff)_6}$\\
${\cc_{348}=(\cc_{309},\ff)_6}$\\
${\cc_{349}=(\cc_{308},\ff)_6}$\\
${\cc_{350}=(\cc_{307},\ff)_6}$\\
${\cc_{351}=(\cc_{306},\ff)_6}$\\
${\cc_{352}=(\cc_{305},\ff)_6}$\\
${\cc_{353}=(\cc_{304},\ff)_6}$\\
${\cc_{354}=(\cc_{303},\ff)_6}$\\
${\cc_{355}=(\cc_{302},\ff)_6}$\\
${\cc_{356}=(\cc_{55}\cc_{50},\ff)_{10}}$\\
${\cc_{357}=(\cc_{50}\cc_{54},\ff)_{10}}$\\
${\cc_{358}=(\cc_{50}\cc_{53},\ff)_{10}}$\\
${\cc_{359}=(\cc_{50}\cc_{52},\ff)_{10}}$\\
${\cc_{360}=(\cc_{50}\cc_{51},\ff)_{10}}$\\
${\cc_{361}=(\cc_{50}^2,\ff)_9}$\\
${\cc_{362}=(\cc_{55}\cc_{49},\ff)_{10}}$\\
\noindent\hspace*{0.2cm}\textbf{Order} 8:\\
${\cc_{363}=(\cc_{315},\ff)_6}$\\

\noindent\textbf{Degree} 12:\\
\noindent\hspace*{0.2cm}\textbf{Order} 0:\\
${\cc_{364}=(\cc_{55}\cc_{78},\ff)_{10}}$\\
${\cc_{365}=(\cc_{55}\cc_{77},\ff)_{10}}$\\
${\cc_{366}=(\cc_{78}\cc_{54},\ff)_{10}}$\\
${\cc_{367}=(\cc_{50}\cc_{83},\ff)_{10}}$\\
${\cc_{368}=(\cc_{82}\cc_{50},\ff)_{10}}$\\
${\cc_{369}=(\cc_{81}\cc_{50},\ff)_{10}}$\\
${\cc_{370}=(\cc_{80}\cc_{50},\ff)_{10}}$\\
${\cc_{371}=(\cc_{79}\cc_{50},\ff)_{10}}$\\
${\cc_{372}=(\cc_{49}\cc_{83},\ff)_{10}}$\\
${\cc_{373}=(\cc_{49}\cc_{82},\ff)_{10}}$\\
${\cc_{374}=(\cc_{46}\cc_{91},\ff)_{10}}$\\
${\cc_{375}=(\cc_{46}\cc_{90},\ff)_{10}}$\\
\noindent\hspace*{0.2cm}\textbf{Order} 2:\\
${\cc_{376}=(\cc_{363},\ff)_8}$\\
${\cc_{377}=(\cc_{55}\cc_{83},\ff)_{10}}$\\
${\cc_{378}=(\cc_{55}\cc_{82},\ff)_{10}}$\\
${\cc_{379}=(\cc_{55}\cc_{81},\ff)_{10}}$\\
${\cc_{380}=(\cc_{55}\cc_{80},\ff)_{10}}$\\
${\cc_{381}=(\cc_{55}\cc_{79},\ff)_{10}}$\\
${\cc_{382}=(\cc_{55}\cc_{78},\ff)_9}$\\
${\cc_{383}=(\cc_{55}\cc_{77},\ff)_9}$\\
${\cc_{384}=(\cc_{83}\cc_{54},\ff)_{10}}$\\
${\cc_{385}=(\cc_{82}\cc_{54},\ff)_{10}}$\\
${\cc_{386}=(\cc_{81}\cc_{54},\ff)_{10}}$\\
${\cc_{387}=(\cc_{80}\cc_{54},\ff)_{10}}$\\
${\cc_{388}=(\cc_{79}\cc_{54},\ff)_{10}}$\\
${\cc_{389}=(\cc_{78}\cc_{54},\ff)_9}$\\
${\cc_{390}=(\cc_{77}\cc_{54},\ff)_9}$\\
${\cc_{391}=(\cc_{83}\cc_{53},\ff)_{10}}$\\
${\cc_{392}=(\cc_{82}\cc_{53},\ff)_{10}}$\\
${\cc_{393}=(\cc_{81}\cc_{53},\ff)_{10}}$\\
${\cc_{394}=(\cc_{80}\cc_{53},\ff)_{10}}$\\
${\cc_{395}=(\cc_{79}\cc_{53},\ff)_{10}}$\\
${\cc_{396}=(\cc_{78}\cc_{53},\ff)_9}$\\
${\cc_{397}=(\cc_{77}\cc_{53},\ff)_9}$\\
${\cc_{398}=(\cc_{83}\cc_{52},\ff)_{10}}$\\
${\cc_{399}=(\cc_{82}\cc_{52},\ff)_{10}}$\\
${\cc_{400}=(\cc_{81}\cc_{52},\ff)_{10}}$\\
${\cc_{401}=(\cc_{80}\cc_{52},\ff)_{10}}$\\
${\cc_{402}=(\cc_{79}\cc_{52},\ff)_{10}}$\\
${\cc_{403}=(\cc_{78}\cc_{52},\ff)_9}$\\
${\cc_{404}=(\cc_{91}\cc_{50},\ff)_{10}}$\\
${\cc_{405}=(\cc_{90}\cc_{50},\ff)_{10}}$\\
\noindent\hspace*{0.2cm}\textbf{Order} 4:\\
${\cc_{406}=(\cc_{363},\ff)_7}$\\
\noindent\hspace*{0.2cm}\textbf{Order} 6:\\
${\cc_{407}=(\cc_{363},\ff)_6}$\\
${\cc_{408}=(\cc_{362},\ff)_4}$\\

\noindent\textbf{Degree} 13:\\
\noindent\hspace*{0.2cm}\textbf{Order} 0:\\
${\cc_{409}=(\cc_{91}\cc_{83},\ff)_{10}}$\\
${\cc_{410}=(\cc_{90}\cc_{83},\ff)_{10}}$\\
${\cc_{411}=(\cc_{89}\cc_{83},\ff)_{10}}$\\
${\cc_{412}=(\cc_{88}\cc_{83},\ff)_{10}}$\\
${\cc_{413}=(\cc_{83}\cc_{87},\ff)_{10}}$\\
${\cc_{414}=(\cc_{83}\cc_{86},\ff)_{10}}$\\
${\cc_{415}=(\cc_{83}\cc_{85},\ff)_{10}}$\\
${\cc_{416}=(\cc_{83}\cc_{84},\ff)_{10}}$\\
${\cc_{417}=(\cc_{91}\cc_{82},\ff)_{10}}$\\
${\cc_{418}=(\cc_{90}\cc_{82},\ff)_{10}}$\\
${\cc_{419}=(\cc_{89}\cc_{82},\ff)_{10}}$\\
${\cc_{420}=(\cc_{88}\cc_{82},\ff)_{10}}$\\
${\cc_{421}=(\cc_{82}\cc_{87},\ff)_{10}}$\\
${\cc_{422}=(\cc_{82}\cc_{86},\ff)_{10}}$\\
${\cc_{423}=(\cc_{82}\cc_{85},\ff)_{10}}$\\
\noindent\hspace*{0.2cm}\textbf{Order} 2:\\
${\cc_{424}=(\cc_{91}^2,\ff)_{10}}$\\
${\cc_{425}=(\cc_{90}\cc_{91},\ff)_{10}}$\\
${\cc_{426}=(\cc_{90}^2,\ff)_{10}}$\\
${\cc_{427}=(\cc_{89}\cc_{91},\ff)_{10}}$\\
${\cc_{428}=(\cc_{89}\cc_{90},\ff)_{10}}$\\
${\cc_{429}=(\cc_{89}^2,\ff)_{10}}$\\
${\cc_{430}=(\cc_{88}\cc_{91},\ff)_{10}}$\\
${\cc_{431}=(\cc_{88}\cc_{90},\ff)_{10}}$\\
${\cc_{432}=(\cc_{88}\cc_{89},\ff)_{10}}$\\
${\cc_{433}=(\cc_{88}^2,\ff)_{10}}$\\
${\cc_{434}=(\cc_{91}\cc_{87},\ff)_{10}}$\\
${\cc_{435}=(\cc_{90}\cc_{87},\ff)_{10}}$\\
${\cc_{436}=(\cc_{89}\cc_{87},\ff)_{10}}$\\
${\cc_{437}=(\cc_{88}\cc_{87},\ff)_{10}}$\\
${\cc_{438}=(\cc_{87}^2,\ff)_{10}}$\\
${\cc_{439}=(\cc_{91}\cc_{86},\ff)_{10}}$\\
\noindent\hspace*{0.2cm}\textbf{Order} 4:\\
${\cc_{440}=(\cc_{408},\ff)_6}$\\
${\cc_{441}=(\cc_{407},\ff)_6}$\\

\noindent\textbf{Degree} 14:\\
\noindent\hspace*{0.2cm}\textbf{Order} 0:\\
${\cc_{442}=(\cc_{97}\cc_{119},\ff)_{10}}$\\
${\cc_{443}=(\cc_{118}\cc_{97},\ff)_{10}}$\\
${\cc_{444}=(\cc_{117}\cc_{97},\ff)_{10}}$\\
${\cc_{445}=(\cc_{116}\cc_{97},\ff)_{10}}$\\
${\cc_{446}=(\cc_{115}\cc_{97},\ff)_{10}}$\\
${\cc_{447}=(\cc_{114}\cc_{97},\ff)_{10}}$\\
${\cc_{448}=(\cc_{113}\cc_{97},\ff)_{10}}$\\
${\cc_{449}=(\cc_{96}\cc_{119},\ff)_{10}}$\\
${\cc_{450}=(\cc_{96}\cc_{118},\ff)_{10}}$\\
${\cc_{451}=(\cc_{117}\cc_{96},\ff)_{10}}$\\
${\cc_{452}=(\cc_{116}\cc_{96},\ff)_{10}}$\\
${\cc_{453}=(\cc_{115}\cc_{96},\ff)_{10}}$\\
${\cc_{454}=(\cc_{113}\cc_{96},\ff)_{10}}$\\
\noindent\hspace*{0.2cm}\textbf{Order} 2:\\
${\cc_{455}=(\cc_{129}\cc_{97},\ff)_{10}}$\\
${\cc_{456}=(\cc_{128}\cc_{97},\ff)_{10}}$\\
${\cc_{457}=(\cc_{127}\cc_{97},\ff)_{10}}$\\
${\cc_{458}=(\cc_{126}\cc_{97},\ff)_{10}}$\\
${\cc_{459}=(\cc_{125}\cc_{97},\ff)_{10}}$\\
${\cc_{460}=(\cc_{124}\cc_{97},\ff)_{10}}$\\
${\cc_{461}=(\cc_{123}\cc_{97},\ff)_{10}}$\\
${\cc_{462}=(\cc_{122}\cc_{97},\ff)_{10}}$\\
${\cc_{463}=(\cc_{121}\cc_{97},\ff)_{10}}$\\
${\cc_{464}=(\cc_{97}\cc_{120},\ff)_{10}}$\\
${\cc_{465}=(\cc_{97}\cc_{119},\ff)_9}$\\
${\cc_{466}=(\cc_{118}\cc_{97},\ff)_9}$\\
${\cc_{467}=(\cc_{117}\cc_{97},\ff)_9}$\\
${\cc_{468}=(\cc_{116}\cc_{97},\ff)_9}$\\
${\cc_{469}=(\cc_{115}\cc_{97},\ff)_9}$\\
${\cc_{470}=(\cc_{114}\cc_{97},\ff)_9}$\\
${\cc_{471}=(\cc_{113}\cc_{97},\ff)_9}$\\

\noindent\textbf{Degree} 15:\\
\noindent\hspace*{0.2cm}\textbf{Order} 0:\\
${\cc_{472}=(\cc_{137}\cc_{129},\ff)_{10}}$\\
${\cc_{473}=(\cc_{136}\cc_{129},\ff)_{10}}$\\
${\cc_{474}=(\cc_{135}\cc_{129},\ff)_{10}}$\\
${\cc_{475}=(\cc_{134}\cc_{129},\ff)_{10}}$\\
${\cc_{476}=(\cc_{133}\cc_{129},\ff)_{10}}$\\
${\cc_{477}=(\cc_{132}\cc_{129},\ff)_{10}}$\\
${\cc_{478}=(\cc_{129}\cc_{131},\ff)_{10}}$\\
${\cc_{479}=(\cc_{129}\cc_{130},\ff)_{10}}$\\
${\cc_{480}=(\cc_{137}\cc_{128},\ff)_{10}}$\\
${\cc_{481}=(\cc_{136}\cc_{128},\ff)_{10}}$\\
${\cc_{482}=(\cc_{135}\cc_{128},\ff)_{10}}$\\
${\cc_{483}=(\cc_{134}\cc_{128},\ff)_{10}}$\\
${\cc_{484}=(\cc_{133}\cc_{128},\ff)_{10}}$\\
${\cc_{485}=(\cc_{132}\cc_{128},\ff)_{10}}$\\
${\cc_{486}=(\cc_{128}\cc_{131},\ff)_{10}}$\\
${\cc_{487}=(\cc_{128}\cc_{130},\ff)_{10}}$\\
${\cc_{488}=(\cc_{137}\cc_{127},\ff)_{10}}$\\
${\cc_{489}=(\cc_{136}\cc_{127},\ff)_{10}}$\\
${\cc_{490}=(\cc_{135}\cc_{127},\ff)_{10}}$\\
\noindent\hspace*{0.2cm}\textbf{Order} 4:\\
${\cc_{491}=(\cc_{137}\cc_{149},\ff)_{10}}$\\

\noindent\textbf{Degree} 16:\\
\noindent\hspace*{0.2cm}\textbf{Order} 0:\\
${\cc_{492}=(\cc_{169}\cc_{149},\ff)_{10}}$\\
${\cc_{493}=(\cc_{168}\cc_{149},\ff)_{10}}$\\
${\cc_{494}=(\cc_{167}\cc_{149},\ff)_{10}}$\\
${\cc_{495}=(\cc_{180}\cc_{137},\ff)_{10}}$\\
${\cc_{496}=(\cc_{179}\cc_{137},\ff)_{10}}$\\
\noindent\hspace*{0.2cm}\textbf{Order} 2:\\
${\cc_{497}=(\cc_{180}\cc_{149},\ff)_{10}}$\\
${\cc_{498}=(\cc_{179}\cc_{149},\ff)_{10}}$\\
${\cc_{499}=(\cc_{137}\cc_{195},\ff)_{10}}$\\

\noindent\textbf{Degree} 17:\\
\noindent\hspace*{0.2cm}\textbf{Order} 0:\\
${\cc_{500}=(\cc_{180}\cc_{195},\ff)_{10}}$\\
${\cc_{501}=(\cc_{180}\cc_{194},\ff)_{10}}$\\
${\cc_{502}=(\cc_{180}\cc_{193},\ff)_{10}}$\\
${\cc_{503}=(\cc_{180}\cc_{192},\ff)_{10}}$\\
${\cc_{504}=(\cc_{195}\cc_{175},\ff)_{10}}$\\

\noindent\textbf{Degree} 18:\\
\noindent\hspace*{0.2cm}\textbf{Order} 0:\\
${\cc_{505}=(\cc_{199}\cc_{225},\ff)_{10}}$\\
\noindent\hspace*{0.2cm}\textbf{Order} 2:\\
${\cc_{506}=(\cc_{195}\cc_{252},\ff)_{10}}$\\

\noindent\textbf{Degree} 19:\\
\noindent\hspace*{0.2cm}\textbf{Order} 0:\\
${\cc_{507}=(\cc_{244}\cc_{252},\ff)_{10}}$\\
${\cc_{508}=(\cc_{244}\cc_{251},\ff)_{10}}$\\

\noindent\textbf{Degree} 21:\\
\noindent\hspace*{0.2cm}\textbf{Order} 0:\\
${\cc_{509}=(\cc_{301}\cc_{314},\ff)_{10}}$\\
${\cc_{510}=(\cc_{301}\cc_{313},\ff)_{10}}$\\

  \end{multicols}
\end{tiny}

\end{document}